\documentclass[10pt]{article}
\usepackage{arxiv}
\usepackage[english]{babel}
\usepackage{amsmath,amsthm}
\usepackage{amsfonts}
\usepackage{amssymb}
\usepackage[prependcaption,textsize=scriptsize]{todonotes}
\usepackage{hyperref}
\usepackage{mathtools}
\usepackage[ruled,vlined]{algorithm2e}
\usepackage{tikz}
\usetikzlibrary{shapes,arrows,backgrounds,fit,positioning}
\usepackage{graphicx}
\usepackage{caption}
\usepackage{subcaption}
\usepackage{threeparttable}
\usepackage{booktabs}
\usepackage{enumerate}
\usepackage[utf8]{inputenc}
\usepackage{xtab}

\usepackage{pgfplots}

\usetikzlibrary{arrows,matrix,positioning, decorations.pathreplacing,calc}
\newtheorem{thm}{Theorem}[section]

\newtheorem{lem}[thm]{Lemma}
\newtheorem{prop}[thm]{Proposition}
\theoremstyle{definition}
\newtheorem{defn}[thm]{Definition}
\theoremstyle{remark}
\newtheorem{rem}[thm]{Remark}
\numberwithin{equation}{section}

\usepackage{array}
\newcolumntype{H}{>{\setbox0=\hbox\bgroup}c<{\egroup}@{}}

\begin{document}

\title{Minimum energy configurations on a toric lattice as a quadratic assignment problem}

\author{Daniel Brosch\thanks{Tilburg University, \texttt{D.Brosch@uvt.nl} This project has received funding from
 the European Union’s Horizon 2020 research and innovation programme under the Marie Sk{\l}odowska-Curie grant agreement No 764759.}
 \And Etienne de Klerk\thanks{Tilburg University, \texttt{E.deKlerk@uvt.nl}}}
\date{\today}%

\maketitle
\begin{abstract}
We consider three known bounds for the quadratic assignment problem (QAP): an eigenvalue, a convex quadratic programming (CQP),
 and a semidefinite programming (SDP) bound. Since the last two bounds were not compared directly before, we prove that the SDP
  bound is stronger than the CQP bound. We then apply these to improve known bounds on a discrete energy minimization problem,
   reformulated as a QAP, which aims to minimize the potential energy between repulsive particles on a toric grid.
    Thus we are able to prove optimality for several configurations of particles and grid sizes, complementing
     earlier results by Bouman, Draisma and Van Leeuwaarden [ {\em SIAM Journal on Discrete Mathematics}, 27(3):1295--1312, 2013].
{\color{black}The semidefinite programs in question are too large to solve without pre-processing,
and we use
     a symmetry reduction method by Parrilo and Permenter [{\em Mathematical Programming}, 181:51--84, 2020] to make computation of the SDP bounds possible.}
\end{abstract}



\section{Introduction}

A quadratic assignment problem in Koopmans-Beckmann form is given by three matrices $A = (a_{ij}),B = (b_{ij}),C = (c_{ij})\in\mathbb{R}^{n\times n}$, and can be written as
\begin{equation*}
  QAP(A,B,C) = \min_{\varphi\in S_n} \left(\sum_{i,j=1}^{n}a_{ij}b_{\varphi(i)\varphi(j)}+\sum_{i=1}^{n}c_{i\varphi(i)}\right),
\end{equation*}
where $S_n$ is the set of all permutations of $n$ elements. If $C=0$, then we shorten the notation to $QAP(A,B)$, {\color{black}and if all data matrices are symmetric, we call the quadratic assignment problem symmetric.}

This is a quadratic optimization problem, which can be seen if we write the objective using permutation matrices:
\begin{equation*}
  \min_{X\in \Pi_n} \langle A,XBX^T\rangle + \langle C,X\rangle,
\end{equation*}
where $\langle X,Y\rangle = \mathrm{tr}(X^TY)$ is the trace inner product, and $\Pi_n$ the set of $n\times n$ permutation matrices.

Because of the very general form of the problem, it is not surprising that it is NP-complete (see for example \S 7.1.7 in \cite{burkard2009assignment}),
which motivates the search for good approximations and bounds; see, e.g., the survey \cite{QAP survey EJOR} and the book \cite{burkard2009assignment} for an overview.
In Section \ref{QAPBoundSection} we describe three such bounds, in both increasing complexity and strength.
The first is a projected eigenvalue bound, which was first introduced in \cite{hadley1992new}, which, similar to
 the eigenvalue bound of \cite{finke1987quadratic}, is based on the eigenvalues of the data matrices.
 The second bound, a convex quadratic programming bound, then improves this bound by adding a convex quadratic
  term to the objective, as introduced in \cite{anstreicher2001new} (see also \cite{anstreicher2001,anstreicher2002}). The third bound, which was introduced in \cite{zhao1998semidefinite} and later reformulated in \cite{povh2009copositive}, is a semidefinite programming relaxation of the quadratic assignment problem. As it is the most complex computationally, it is natural to expect it to be stronger than the two other bounds, which we prove in our first main result Theorem \ref{SDPBetterProj}.

In Section \ref{EnergyMinSection} we then apply the three bounds to a discrete energy minimization problem. It was first described in \cite{taillard1995comparison} as the problem of printing a particular {\color{black}shade of grey}, by repeating the same tile of black and white squares in all directions. Other applications from physics is the search for ground states of a two-dimensional repulsive lattice gas at zero temperature (\cite{watson1997repulsive}), and more generally the Falicov-Kimball model (\cite{falicov1969simple,kennedy1994some}), which is relevant for modelling valence fluctuations in transition metal oxides, binary alloys and high-temperature super-conductors (\cite{watson1997repulsive}).

To get a distribution of black and white tiles as equal as possible, it is natural to view {\color{black}the problem of printing a shade of grey} as a problem of minimizing the potential energy between repulsive particles on a toric grid. This problem can then be reformulated as a quadratic assignment problem, which allows us to apply the three bounds of Section \ref{QAPBoundSection} to this problem. We will see in Proposition \ref{EigenIsProjBound} and Proposition \ref{PBisQPB} that both the projected eigenvalue bound, as well as the convex quadratic programming bound, coincide with an eigenvalue bound for this problem introduced in \cite{bouman2013energy}.
{\color{black}
In Section \ref{sec:torusQAPreduce} we describe the technique one may use to calculate the semidefinite programming bound,
 which involves a symmetry reduction of the problem to a more manageable size. Our approach is based on the recent Jordan reduction method of Parrilo and Permenter \cite{permenter2016dimension}.
}
Finally, in Section \ref{sec:numerical results} we present numerical results on the bounds for instances on different grid sizes,
including the semidefinite programming bound after Jordan reduction,  and thus prove  optimality of certain grid arrangements. In this way
we extend  earlier results by Bouman, Draisma and Van Leeuwaarden \cite{bouman2013energy}.

\section{Bounds for quadratic assignment problems}\label{QAPBoundSection}
In this section we will consider three different bounds for QAPs, of increasing computational complexity. These are then compared to each other in Section \ref{BoundComparisonSubSection}, and applied to an energy minimization problem in Section \ref{EnergyMinSection}.
\subsection{Projected eigenvalue bound}
The first bound relevant for this paper is the projected eigenvalue bound, which was introduced in \cite{hadley1992new}, a stronger variant of the eigenvalue bound for QAP (see \cite{finke1987quadratic}), which is based on projecting the matrices into a space the same dimension as the span of the permutation matrices.

{\color{black}We denote the all-ones vector with $e$, and the elements of the canonical basis as $e_i$.}

\begin{prop}[\cite{hadley1990bounds},\cite{hadley1992new}, cf. Prop. 7.23 in \cite{burkard2009assignment}]\label{PBBound}
  Let $V$ be the $n\times (n-1)$ matrix, of which the columns form an orthonormal basis of the orthogonal complement of the all-ones vector $e$. Define $\tilde{A} \coloneqq V^T AV$, $\tilde{B} \coloneqq V^T BV$, and collect their eigenvalues in the vectors $\lambda_{\tilde{A}}$ and $\mu_{\tilde{B}}$ respectively. Set $D\coloneqq\frac{2}{n}Aee^TB$. The \emph{projected eigenvalue bound} for the symmetric $\mathrm{QAP}(A,B)$ is given by
  \begin{equation}
    \mathrm{PB}(A,B) \coloneqq \langle \lambda_{\tilde{A}},\mu_{\tilde{B}}\rangle^-+\min_{\varphi\in S_n}\sum_{i=1}^{n}d_{i\varphi(i)}-\frac{(e^TAe)(e^TBe)}{n^2},
  \end{equation}
  where $\langle x ,y\rangle^-=\min_{\varphi\in S_n} \sum_{i=1}^{n} x_{\varphi(i)}y_i$. One then has $PB(A,B)\leq QAP(A,B)$.
\end{prop}
One may calculate $\mathrm{PB}(A,B)$ by sorting $\lambda_{\tilde{A}}$ and $\mu_{\tilde{B}}$ to compute $\langle \lambda_{\tilde{A}},\mu_{\tilde{B}}\rangle^-$ (see Proposition 5.8 in \cite{burkard2009assignment}) and solving one linear assignment problem $\min_{\varphi\in S_n}\sum_{i=1}^{n}d_{i\varphi(i)}$.
\subsection{QP bound}
The second bound we consider is a convex quadratic programming (CQP) bound, introduced in \cite{anstreicher2001new}, which is based on the same projection as the bound in Proposition \ref{PBBound}. We will see that it is at least as good as the projected eigenvalue bound. Here we relax $X\in\Pi_n$ to $Xe=X^Te=e$ and $X\geq 0$, i.e. we optimize over doubly stochastic matrices instead of permutation matrices.

{\color{black}In the following $I_n$ and $J_n$ denote the identity and all-ones matrices respectively of size $n\times n$, $E_{ij}$ denotes the $n\times n$ matrix with a single one at position $(i,j)$, and $\otimes$ the Kronecker product.}

  Hadley, Rendl and Wolkowicz (\cite{hadley1990bounds},\cite{hadley1992new}) observed that every doubly stochastic matrix can be written as
  \begin{equation*}
    X=\frac{1}{n}ee^T + VYV^T,
  \end{equation*}
  where $V$ is the $n\times (n-1)$ matrix of which the columns form an orthonormal basis of the orthogonal complement of $e$, as before. We have $V^TV = I_{n-1}$, $VV^T = I_n - \frac{1}{n}ee^T$ and $Y = V^TXV$. As before, we set $\tilde{A}=V^TAV$ and $\tilde{B}=V^TBV$, and collect their eigenvalues in the vectors $\lambda_{\tilde{A}}$ and $\mu_{\tilde{B}}$.

  In Section 3 of \cite{anstreicher2001new}, Anstreicher and Brixius introduce the following CQP bound for quadratic assignment problems.
  \begin{prop}[\cite{anstreicher2001new}]\label{QPBBound}
    Let $A$ and $B$ be symmetric matrices of size $n\times n$, {\color{black}and define the pair $(S^*,T^*)$ to be any optimal solution  of the problem}
    \begin{equation*}
    \max \left\lbrace\mathrm{tr}(S+T)\colon \tilde{B}\otimes\tilde{A}-I_n\otimes S-T\otimes I_n\succcurlyeq 0\right\rbrace,
  \end{equation*}
  so the matrix $\hat{Q}\coloneqq \tilde{B}\otimes\tilde{A}-I_n\otimes S^*-T^*\otimes I_n\succcurlyeq 0$ is positive semidefinite, and $\mathrm{tr}(S^*+T^*) = \langle \lambda_{\tilde{A}},\mu_{\tilde{B}}\rangle^-$. Then we get a convex quadratic bound for $QAP(A,B)$, which is at least as good as $PB(A,B)$, by
   \begin{align}
    QPB(A,B) \coloneqq \min \enspace& y^T\hat{Q}y +  \langle \lambda_{\tilde{A}},\mu_{\tilde{B}}\rangle^- + \frac{2}{n}\mathrm{tr}\left(BJ_nAX\right) - \frac{(e^TAe)(e^TBe)}{n^2},\\
    \text{s.t. }\enspace&X\geq 0 \text{ is doubly stochastic},\nonumber\\
    &X=\frac{1}{n}ee^T + VYV^T,\nonumber\\
    &y = \mathrm{vec}(Y).\nonumber
  \end{align}
  \end{prop}
  In other words, one always has $PB(A,B)\leq QPB(A,B)$.

  One may compute $\mathrm{QPB(A,B)}$ by solving a linear assignment problem to obtain $\hat{Q}$, and then solving a CQP in $\mathcal{O}(n^2)$ variables. For details, see Section 4 in \cite{anstreicher2001new}.

\subsection{SDP bound}
 The following semidefinite relaxation for $\mathrm{QAP}(A,B,C)$ was studied by Povh and Rendl \cite{povh2009copositive}, which is equivalent
 to an earlier bound by Zhao, Karisch, Rendl and Wolkowicz \cite{zhao1998semidefinite}:
    \begin{align}\label{QAPSDP}
      SDPQAP(A,B,C) \coloneqq \min\enspace & \langle B\otimes A + \mathrm{Diag}(\mathrm{vec}(C)),Y\rangle\\
      \mathrm{s.t.}\enspace & \langle I_n\otimes E_{jj},Y\rangle=1 \text{ for }j=1,\ldots,n,\nonumber\\
      & \langle E_{jj}\otimes I_n,Y\rangle=1 \text{ for }j=1,\ldots,n,\nonumber\\
      & \langle I_n\otimes (J_n-I_n)+(J_n-I_n)\otimes I_n,Y\rangle =0,\nonumber \\
      & \langle J_{n^2},Y\rangle = n^2,\nonumber\\
      & Y\in S^{n^2}_+ \cap \mathbb{R}^{n^2\times n^2}_{\geq 0},\nonumber
    \end{align}
    where $A,B,C\in \mathbb{R}^{n\times n}$ and $A$ and $B$ are symmetric, {\color{black}and $S^{n}_+$ denotes the cone of positive semidefinite matrices of size $n\times n$}. We write $SDPQAP(A,B)$ if $C=0$. {\color{black}In general, if the dimension of matrices is clear, we write $X\succcurlyeq 0$ instead of $X\in S^n_+$ to denote positive semidefinite matrices.}

    The bound $SDPQAP(A,B,C)$ is expensive to compute, as it involves an SDP with doubly nonnegative matrix variables of order $n^2\times n^2$.
\subsection{Bound comparison}\label{BoundComparisonSubSection}
The three bounds $PB(A,B)$, $QPB(A,B)$ and $SDPQAP(A,B)$ increase in computational complexity, hence we would expect that the bounds do get better accordingly. As it turns out, we can show the expected order of the bound quality.
\begin{thm}\label{SDPBetterProj}
For symmetric matrices $A$ and $B$ we have
  \begin{equation}
    PB(A,B) \leq QPB(A,B) \leq SDPQAP(A,B) \leq QAP(A,B).
  \end{equation}
\end{thm}

{\color{black} Proving this bound is quite technical. The outermost inequalities are known, hence we only need to show the inequality in the middle. Instead of doing this directly, we first introduce another SDP-based bound, which lies in between the bounds $QPB$ and $SDPQAP$. This bound, which we will call $SDPPB$, is based on the same projection used for $PB$ and $QPB$.}

 Again we start with the observation that for every doubly stochastic {\color{black}$n\times n$-matrix} $X$ we can always find an {\color{black}$n-1\times n-1$-matrix} $Y$ with
  \begin{equation}
    X=\frac{1}{n}ee^T+VYV^T,
  \end{equation}
  where $V{\color{black}=(v_1\vert\ldots\vert v_{n-1})}$ is the $n\times (n-1)$ matrix of which the columns form an orthonormal basis of the orthogonal complement of the all one vector $e$. Set $y=\mathrm{vec}(Y)$, the vector we obtain by gluing the columns of $Y$ together. The idea of this bound is now to relax $U = yy^T$ to be a {\color{black}positive} semidefinite matrix with certain constraints, {\color{black}and to relax $y=\mathrm{vec}(Y)$ to a variable $u$. Since $yy^T-yy^T=0\succcurlyeq 0$, we add the constraint that $U-uu^T\succcurlyeq 0$.}

   We can rewrite the objective function of the QAP at $X=\frac{1}{n}ee^T+VYV^T$ as:
     \begin{align}
    \mathrm{tr}(AXBX^T) &= \mathrm{tr}\left(A\left(\frac{1}{n}ee^T + VYV^T\right)B\left(\frac{1}{n}ee^T + VYV^T\right)^T\right) \\
     &= \mathrm{tr}\left(\frac{1}{n^2}Aee^TBee^T\right) + \mathrm{tr}\left(\frac{1}{n}Aee^TBVY^TV^T\right)\\
     &\quad+ \mathrm{tr}\left(\frac{1}{n}AVYV^TBee^T\right)+\mathrm{tr}\left(\frac{1}{n}AVYV^TBVY^TV^T\right)\\
     &=\mathrm{tr}\left(\tilde{A}Y\tilde{B}Y^T\right) + \frac{2}{n}\mathrm{tr}\left(BJ_nAVYV^T\right) + \frac{1}{n^2}(e^TAe)(e^TBe)\\
    &= \langle \tilde{B}\otimes\tilde{A}, yy^T\rangle + \frac{2}{n}\mathrm{vec}\left(V^TBJ_nAV\right)^Ty + \frac{1}{n^2}(e^TAe)(e^TBe).
  \end{align}
  Thus we can write it as linear function of $yy^T$ and $y$, which we relax to $U$ and $u$ respectively. To make this bound at least as good as the convex quadratic bound $QPB(A,B)$, we have to add more conditions, which follow from $U=yy^T=\mathrm{vec}(V^TXV)\mathrm{vec}(V^TXV)^T$. We have, for all $1\leq i,j\leq n-1$, that
  \begin{align}
    \langle I\otimes E_{ij},U\rangle &= \mathrm{tr}\left(E_{ij}V^TXV I V^TX^TV\right)\\
    &= \mathrm{tr}\left(v_jv_i^T X \left(I-\frac{1}{n}J\right)X^T\right)\\
    & = \mathrm{tr}\left(v_jv_i^T\right) - \frac{1}{n}\mathrm{tr}\left(v_jv_i^TJ\right)\\
    &= \mathrm{tr}\left(v_i^Tv_j\right) - \frac{1}{n}\mathrm{tr}\left(e^Tv_jv_i^Te\right)= \delta_{ij},
  \end{align}
  and analogously we can show that $\langle E_{ij}\otimes I,U\rangle = \delta_{ij}$ as well. {\color{black}Here $\delta_{ij}$ denotes the Kronecker-Delta, which is one if $i=j$, and zero otherwise.} Finally, the property that $X=\frac{1}{n}ee^T+VYV^T$ is nonnegative is equivalent to $(V\otimes V)y\geq -\frac{1}{n}e\otimes e$.
  {\color{black}
  \begin{prop} With $\tilde{A},\tilde{B}$ and $V$ as defined in Proposition \ref{PBBound}, we obtain a bound for $QAP(A,B)$ by:
    \color{black}
    \begin{align}
    SDPPB(A,B) \coloneqq \min \enspace&\langle \tilde{B}\otimes \tilde{A}, U\rangle  + \frac{2}{n}\mathrm{vec}(V^TBJAV)^T u+\frac{1}{n^2}(e^TAe)(e^TBe)\label{SDPPB}\\
    \text{s.t. } & \begin{pmatrix}
                     1 & u^T \\
                     u & U
                   \end{pmatrix}\succcurlyeq 0,\nonumber\\
    & \langle E_{ij}\otimes I_{n-1}, U\rangle = \delta_{ij} \quad\forall i,j=1,\ldots,n-1,\nonumber\\
    & \langle I_{n-1}\otimes E_{ij}, U\rangle = \delta_{ij} \quad\forall i,j=1,\ldots,n-1,\nonumber\\
    & (V\otimes V) u \geq -\frac{1}{n}e\otimes e.\nonumber
  \end{align}
  \end{prop}
  }

{\color{black}
\begin{lem}\label{SDPPB_better_than_QPB}
  $$SDPPB(A,B) \geq QPB(A,B).$$
\end{lem}
\begin{proof}
\color{black}
    Let $(U,u)$ be an optimal solution of $SDPPB(A,B)$. Then we construct a feasible solution for $QPB(A,B)$ by setting $y=u$, since $(V\otimes V)u + \frac{1}{n}ee^T\geq 0$, and thus $X=\frac{1}{n}ee^T+VYV^T\geq 0$ for $\mathrm{vec}(Y)=y$.
     The matrix $X$ is doubly stochastic, since $(e\otimes e_i)^T(V\otimes V)u = (e^TV\otimes e_i^TV)u = 0$ and $(e_i\otimes e)^T(V\otimes V)u =0$, so adding $VYV^T$ to the doubly stochastic matrix $\frac{1}{n}ee^T$ results in another doubly stochastic matrix.

  {\color{black}In the following $\hat{Q}$ is as defined in Proposition \ref{QPBBound}.} By the Schur complement theorem we know that $U-uu^T\succcurlyeq 0$, hence we have that
  \begin{align}
    u^T\hat{Q}u & \leq \langle \hat{Q},U\rangle \\
     &= \langle \tilde{B}\otimes \tilde{A}, U\rangle - \langle I\otimes S^*, U\rangle-\langle T^*\otimes I, U\rangle\\
     &= \langle \tilde{B}\otimes \tilde{A}, U\rangle - \sum_{i,j=1}^{n-1} S^*_{ij} \langle I\otimes E_{ij},U\rangle-\sum_{i,j=1}^{n-1} T^*_{ij} \langle E_{ij}\otimes I,U\rangle\\
     &= \langle \tilde{B}\otimes \tilde{A}, U\rangle - \mathrm{tr}(S^*) - \mathrm{tr}(T^*)\\
     &= \langle \tilde{B}\otimes \tilde{A}, U\rangle - \langle \lambda_{\tilde{A}},\mu_{\tilde{B}}\rangle^-.
  \end{align}
  Thus we can compare the two bounds by
  \begin{align}
    QPB(A,B)&\leq y^T\hat{Q}y+\langle \lambda_{\tilde{A}},\mu_{\tilde{B}}\rangle +\frac{2}{n}\mathrm{tr}(BJAX)-\frac{1}{n^2}(e^TAe)(e^TBe)  \\
     & \leq \langle \tilde{B}\otimes \tilde{A}, U\rangle+\frac{2}{n}\mathrm{tr}(BJAX)-\frac{1}{n^2}(e^TAe)(e^TBe)\\
     & = \langle \tilde{B}\otimes \tilde{A}, U\rangle  + \frac{2}{n}\mathrm{vec}(V^TBJAV)^T u+\frac{1}{n^2}(e^TAe)(e^TBe)\\
     &= SDPPB(A,B).
  \end{align}
\end{proof}

}

{\color{black} To prove the other inequality, we make use of a lemma of Povh and Rendl \cite{povh2009copositive}. They give an alternative description of the feasible set of $QAPSDP$ in terms of blocks of $Y$. For this we split the $n\times n$-matrix-variable $Y$ of \eqref{QAPSDP} into $n^2$ blocks of size $n\times n$, which we call $Y^{(ij)}$. We write $Y=\left[Y^{(ij)}\right]_{1\leq i,j\leq n}$, and use similar notation for other block-matrices.

\begin{lem}[Lemma 6 in \cite{povh2009copositive}]\label{PohvRendlList}
\color{black}
  An $Y=\left[Y^{(ij)}\right]_{1\leq i,j\leq n}\in S^{n^2}_+, Y\geq 0$ is feasible for \eqref{QAPSDP} if and only if
  \begin{enumerate}[$(i)$]
    \item\label{itm:first} $\langle I_n\otimes (J_n-I_n)+(J_n-I_n)\otimes I_n,Y\rangle =0$,
    \item\label{itm:second} $\mathrm{tr}\left(Y^{(ii)}\right)=1\quad\forall i,\qquad \sum_{i=1}^{n}\mathrm{diag}\left(Y^{(ii)}\right)=e$,
    \item\label{itm:third} $Y^{(ij)}e=\mathrm{diag}\left(Y^{(jj)}\right)\quad\forall i,j$,
    \item\label{itm:fourth} $\sum_{i=1}^{n}Y^{(ij)}=e\enspace\mathrm{diag}\left(Y^{(jj)}\right)^T$.
  \end{enumerate}
\end{lem}

}

{\color{black}
\begin{lem}\label{SDPPB_worse_than_SDPQAP}
  $$SDPPB(A,B) \leq SDPQAP(A,B).$$
\end{lem}
\begin{proof}
\color{black}
 With the properties of Lemma \ref{PohvRendlList} we can show that we get a feasible solution for $SDPPB(A,B)$ from a feasible solution $Y$ of $SDPQAP(A,B)$ by setting $U = (V^T\otimes V^T) Y (V\times V)$ and $u=(V^T\otimes V^T)y$, which is the transformation to a Slater-feasible variant of $SDPQAP$ (see e.g. the thesis of Uwe Truetsch \cite{ff97cfeaa5da41e083160c27a848c6ec}). Similarly to $Y$, we can split $U$ into $(n-1)^2$ blocks of size $(n-1)\times (n-1)$, which we call $U^{(ij)}$. We get an explicit formula for these blocks in terms of the $Y^{(ij)}$, if we see $V\otimes V$ as $n(n-1)$ blocks of size $n\times (n-1)$, since then all block sizes are compatible with multiplication.
  \begin{align}
    U &= (V^T\otimes V^T) Y (V\times V),\\
      &= (V^T\otimes V^T)\left(\sum_{k=1}^{n}Y^{(ik)}V_{kj}V \right)_{\substack{1\leq i\leq n\\ 1\leq j\leq n-1}},\\
      &= \left(\sum_{l=1}^{n}V_{li}V^T\sum_{k=1}^{n}Y^{(lk)}V_{kj}V \right)_{1\leq i,j\leq n-1},
  \end{align}
  hence
  \begin{equation}
    U^{(ij)} = \sum_{l,k=1}^{n}V_{li}V_{kj}V^TY^{(lk)}V.
  \end{equation}

  We can now use $\eqref{itm:first}$-$\eqref{itm:fourth}$ to derive some properties of $U$. First note that by $\eqref{itm:first}$ and $\eqref{itm:second}$ we know that $\mathrm{tr}(Y^{(ij)})=\delta_{ij}$, and by $\eqref{itm:second}$ and $\eqref{itm:third}$ that $\mathrm{tr}(Y^{(ij)}J)=1$. Hence we see that
  \begin{align}
    \langle E_{ij}\otimes I_{n-1}, U\rangle &= \mathrm{tr}(U^{(ij)}), \\
    &= \mathrm{tr}\left(\sum_{l,k=1}^{n}V_{li}V_{kj}V^TY^{(lk)}V\right),\\
    &= \sum_{l,k=1}^{n}V_{li}V_{kj}\mathrm{tr}\left(Y^{(lk)}\left(I_n-\frac{1}{n}J_n\right)\right),\\
    &= \sum_{l=1}^{n}V_{li}V_{lj} - \frac{1}{n}\sum_{l,k=1}^{n}V_{li}V_{kj},\\
    &= v_i^Tv_j - 0=\delta_{ij}.
  \end{align}
  Similarly we can use $\sum_{i=1}^{n-1}V_{li}V_{ki}=(VV^T)_{lk}=\delta_{lk}-\frac{1}{n}$, $\eqref{itm:first}$, $\eqref{itm:second}$ and $\eqref{itm:fourth}$ to show that
  \begin{align}
    \langle I_{n-1}\otimes E_{ij}, U\rangle &= \left(\sum_{i=1}^{n-1}U^{(ii)}\right)_{ij}, \\
    &= \left(\sum_{i=1}^{n-1}\sum_{l,k=1}^{n}V_{li}V_{ki}V^TY^{(lk)}V\right)_{ij},\\
    &= \left(\sum_{l=1}^{n}V^TY^{(ll)}V - \frac{1}{n}\sum_{l,k=1}^{n}V^TY^{(lk)}V\right)_{ij},\\
    &= \left(V^TV-\frac{1}{n}V^TJV\right)_{ij},\\
    &=\left(I_{n-1}-0\right)_{ij}= \delta_{ij}.
  \end{align}

  To construct a feasible $u$ with the objective value we need, we use that we can add $Y-yy^T\succcurlyeq 0$, $y=\mathrm{diag}(Y)$ to \eqref{QAPSDP} without changing the optimal value of $SDPQAP$. With an optimal solution $(Y,y)$ we thus set $(U,u)=\left((V^T\otimes V^T)Y(V\otimes V), (V^T\otimes V^T)y\right)$. With $\eqref{itm:second}$ we see that
  \begin{align}
    (V\otimes V)u &= (V\otimes V)(V^T\otimes V^T)y,\\
    &= (I-\frac{1}{n}J)\otimes (I-\frac{1}{n}J)y,\\
    &= y - \frac{1}{n}(J\otimes I) - \frac{1}{n}(I\otimes J) + \frac{1}{n^2}Jy,\\
    &= y-\frac{1}{n}e\geq -\frac{1}{n}e.
  \end{align}

  Since $Y$ and $Y-yy^T$ are positive semidefinite, the matrices $(V^T\otimes V^T)Y (V\otimes V)=U$ and $(V^T\otimes V^T)(Y-yy^T)(V\otimes V) = U-uu^T$ are positive semidefinite as well, and thus feasible for $SDPPB(A,B)$. What remains to be seen is that the objective values of the two programs are the same.
  \begingroup
  \allowdisplaybreaks
  \begin{align}
    \langle \tilde{B}\otimes \tilde{A}, U\rangle & = \mathrm{tr}((V\otimes V)(V^T\otimes V^T)(B\otimes A)(V\otimes V)(V^T\otimes V^T)Y),\\
    &= \mathrm{tr}(((I-\frac{1}{n}J)\otimes(I-\frac{1}{n}J))(B\otimes A)((I-\frac{1}{n}J)\otimes(I-\frac{1}{n}J))Y),\\
    &= \mathrm{tr}((B\otimes A)(Y-\frac{1}{n}ey^T)((I-\frac{1}{n}J)\otimes(I-\frac{1}{n}J))),\\
    &= \mathrm{tr}((B\otimes A)(Y-\frac{1}{n}ey^T-\frac{1}{n}ye^T+\frac{1}{n^2}J)),\\
    &= \langle B\otimes A, Y\rangle - \frac{2}{n}e^T(B\otimes A)y + \frac{1}{n^2}(e^TAe)(e^TBe),
  \end{align}
  and
  \begin{align}
    \frac{2}{n}\mathrm{vec}(V^TBJAV)^T u& = \frac{2}{n}\mathrm{vec}(V^TBJAV)^T(V^T\otimes V^T)y,\\
    &= \frac{2}{n}\mathrm{vec}(VV^TBJAVV^T)^Ty,\\
    &= \frac{2}{n}\mathrm{vec}((I-\frac{1}{n}J)BJA(I-\frac{1}{n}J))^Ty,\\
    &= \frac{2}{n}\mathrm{vec}(BJA)^Ty - \frac{2}{n^2}\mathrm{vec}(BJA)^T(I\otimes J)y \\
    &\quad- \frac{2}{n^2}\mathrm{vec}(BJA)^T(J\otimes I)y+\frac{2}{n^3}\mathrm{vec}(BJA)^TJy,\\
    &= \frac{2}{n}\mathrm{vec}(BJA)^Ty - \frac{2}{n^2}(e^TAe)(e^TBe),
  \end{align}
  thus
  \begin{align}
    \langle \tilde{B}\otimes \tilde{A}&, U\rangle  + \frac{2}{n}\mathrm{vec}(V^TBJAV)^T u+\frac{1}{n^2}(e^TAe)(e^TBe)= \langle B\otimes A, Y\rangle.
  \end{align}
  \endgroup
  Here we used properties $\eqref{itm:first}$-$\eqref{itm:fourth}$, that $\mathrm{vec}(ABC)=(C^T\otimes A)\mathrm{vec}(B)$ and that $A$ and $B$ are symmetric.
\end{proof}

}

\begin{proof}[Proof of Theorem \ref{SDPBetterProj}]
  The only inequality we have to show is $QPB(A,B) \leq SDPQAP(A,B)$, since the leftmost inequality was shown in \cite{anstreicher2001new}. By Lemma \ref{SDPPB_better_than_QPB} and Lemma \ref{SDPPB_worse_than_SDPQAP} we have
  \begin{equation*}
    QPB(A,B)\leq SDPPB(A,B)\leq SDPQAP(A,B)
  \end{equation*}
  and the theorem follows.
\end{proof}

\begin{rem}
While it was expected that $SDPQAP(A,B)$, which has $\mathcal{O}(n^4)$ linear inequality constraints, is better than the projected eigenvalue bound $PB(A,B)$, it was less so for the bound $SDPPB(A,B)$ introduced during the proof of Theorem \ref{SDPBetterProj}, since it only has $\mathcal{O}(n^2)$ linear inequality constraints.
\end{rem}


\section{Energy minimization on a toric grid as a QAP}\label{EnergyMinSection}
We generalize a problem described by Taillard in \cite{taillard1995comparison}, which models the problem of printing a certain shade of grey with only black and white squares ("pixels"). An example of these problems is included in the QAPLIB dataset \cite{burkard1997qaplib}, namely Tai64c.

The goal is to print a particular shade of grey with a given density $m/n$ (ratio of black to total squares),
which is done by repeating a grid of $n=n_1\times n_2$ square cases, exactly $m$ of which are black.
We want the cases to be as regular as possible, so it is natural to model this as an energy minimization problem,
with repulsive particles corresponding to the black squares. {\color{black}We define $[k] := \{1,2,\ldots,k\}$.}
If we have two particles (i.e.
 black squares) at locations $(x_1,y_1)\in [n_1]\times [n_2]$ and $(x_2,y_2)\in [n_1]\times [n_2]$, the potential energy between them is inverse to
their distance, where the distance is given by the shortest path metric on the toric grid, {\color{black}which also known as \emph{Lee metric}}.
We use this grid, since we wish to tile the plane with the $n_1\times n_2$-rectangles, as can be seen in Figure \ref{GridToTorus}.
The potential energy associated with two repulsive particles at respective positions $(x_1,y_1)\in [n_1]\times [n_2]$ and $(x_2,y_2)\in [n_1]\times [n_2]$ is
\begin{equation}
  f_{(x_1, y_1),(x_2, y_2)} = \frac{1}{d_{\mathrm{Lee}}((x_1, y_1),(x_2, y_2))},
\end{equation}
if the coordinates are different, where
\begin{equation}
  d_{\mathrm{Lee}}((x_1, y_1),(x_2, y_2)) = \min(\vert x_1-x_2\vert, n_1 - \vert x_1-x_2\vert )+\min(\vert y_1-y_2\vert, n_2-\vert y_1-y_2\vert)
\end{equation}
is the Lee distance, given by the shortest path metric on the toric grid. We set $f_{i,i} = 0$.

\iftrue

\begin{figure}[h]
\begin{center}
\includegraphics[width=0.75\textwidth]{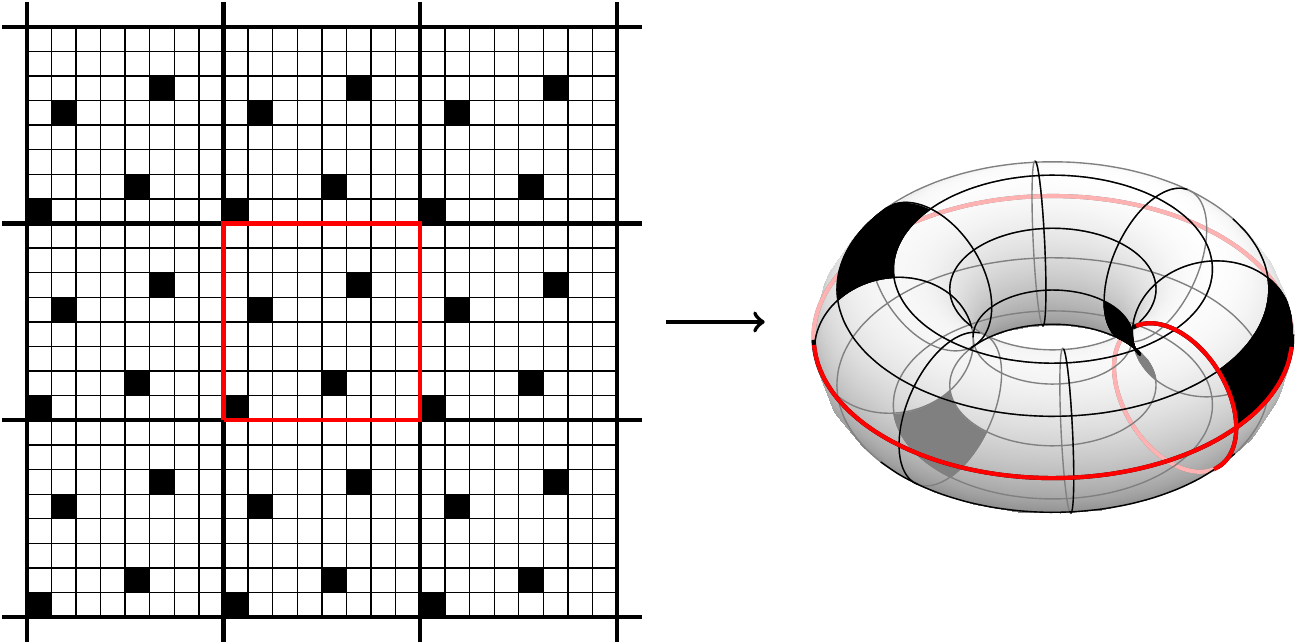}
 \caption{Example of a $n_1\times n_2=8\times 8$ grid tiling with $m=4$, and the corresponding toric interpretation of the $8\times 8$ grid.}\label{GridToTorus}
\end{center}
\end{figure}

\fi

We can then formulate this problem as QAP with matrices {\color{black}$A=(a_{ij}),B=(b_{ij})\in \mathbb{R}^{n\times n}$},
 indexed by grid points $i=(x_i,y_i) \in [n_1]\times [n_2]$, $j=(x_j,y_j)\in [n_1]\times [n_2]$, given by
\begin{align}
  {\color{black}a_{ij}}=\begin{cases}
           1, & \mbox{if } i,j\leq m \\
           0, & \mbox{otherwise}.
         \end{cases}, & \qquad {\color{black}b_{ij}} = f_{i,j} = f_{(x_i,y_i),(x_j,y_j)}.\label{ABDef}
\end{align}
In the definition of $A$ we compared a grid point to an integer $m$. The ordering of the grid points does not matter for the optimal value
 or the symmetry reduction, so it is enough to assume that we have any fixed ordering of the indices, i.e.\ we
  may associate $[n]$ with $[n_1]\times [n_2]$ and will write $[n]=[n_1]\times [n_2]$ when convenient.
 {\color{black} We may also assume  that the ordering is such that  the nonzero elements of the matrix $A$ are given by the $m\times m$-block in the upper left corner, so that}
\begin{equation}\label{EnergyMinQAP}
  \min_{\pi\in S_n} \sum_{i,j=1}^{n} a_{ij}b_{\pi(i)\pi(j)} =  \min_{\pi\in S_n} \sum_{i,j\leq m}b_{\pi(i)\pi(j)} = \min_{\substack{T\subseteq [n_1]\times [n_2] \\ \vert T\vert = m}} \sum_{a,b\in T} f_{a,b}
\end{equation}

Note that the QAP has dimension $n = n_1\times n_2$, and its semidefinite relaxation has dimension $n^2$, which is already 4096 on an $8\times 8$ grid.

To reduce the number of cases one has to look at, we can show a that selecting the complement of an optimal solution leads to another optimal solution.

\begin{prop}\label{InvertArrangement}
  Consider a $[n]=[n_1]\times [n_2]$ grid, and a function $f\colon [n]\times [n]$ with {\color{black}$\sum_{j\in [n]}f(i,j) = \sum_{j\in [n]}f(j,i) = c$ for all $i\in [n]$ and a constant $c\in \mathbb{R}$ }. Then, if $T\subseteq [n]$ minimizes
  \begin{equation*}
    \min_{\substack{T\subseteq [n_1]\times [n_2] \\ \vert T\vert = m}} \sum_{a,b\in T} f_{a,b},
  \end{equation*}
  then $S=[n]\setminus T$ minimizes
  \begin{equation*}
    \min_{\substack{S\subseteq [n_1]\times [n_2] \\ \vert S\vert = n-m}} \sum_{a,b\in S} f_{a,b}.
  \end{equation*}
\end{prop}
\begin{proof}{\color{black}
  We can rewrite the objective function as
  \begin{align*}
    \sum_{a,b\in T} f_{a,b} &= \sum_{a,b\in [n]} f_{a,b} - \sum_{\substack{a\notin T\\b\in [n]}} f_{a,b}- \sum_{\substack{a\in [n] \\b\notin T}} f_{a,b} + \sum_{\substack{a\notin T\\b\notin t}} f_{a,b} \\
    & = c n - 2c(n-m) + \sum_{a,b\in S} f_{a,b}.
  \end{align*}
  Since the term $c n - 2c(n-m)$ is independent of $T$ and $S$, minimizing $\sum_{a,b\in T} f_{a,b}$ is equivalent to minimizing $\sum_{a,b\in S} f_{a,b}$.}
\end{proof}

\subsection{Eigenvalue bound of Bouman, Draisma and Leeuwaarden}
Problem \eqref{EnergyMinQAP} was considered before in \cite{bouman2013energy}, specifically for the case of the Lee-metric, and $m=\frac{n}{2}$ (which we can see as special case of our variant). They took a look at a different relaxation of the problem, which they call \emph{fractional total energy}
\begin{align}
  \min\enspace & x^TBx \label{FractTotalEnergy}\\
    \mathrm{s.t. }\enspace& x^Tx = x^Te = m,\nonumber
\end{align}
where {\color{black}$B=(b_{i,j})$ is the matrix of potential energies between grid points $i$ and $j$, as defined in \ref{ABDef}}.
 {\color{black}In the relaxation \eqref{FractTotalEnergy}, the discrete variables $x_i$ correspond to particle positions, relaxed to continuous values}.
  Thus if $x$ is the characteristic vector of a subset of $m$ {\color{black}particles}, then it is exactly the potential energy of the set of particles.
 {\color{black}  The optimal solution of the relaxation \eqref{FractTotalEnergy} may be expressed in terms of the eigenvalues of $B$ as follows.}
{\color{black}\begin{prop}[Proposition 2.5. in \cite{bouman2013energy}]\label{FractEnergyOpt}
  Let $\lambda_{\min}$ be the smallest eigenvalue of $B$, and $\lambda_1$ the eigenvalue of $B$ corresponding
   to $e$. Then the set of optimal solutions of the  minimization problem \eqref{FractTotalEnergy} consists of all
   vectors of the form $\frac{m}{n}e+y$, where $y$ belongs to the eigenspace of $B$ with eigenvalue $\lambda_{\min}$,
   is perpendicular to $e$, and satisfies $y^Ty = m-\frac{m^2}{n}$.

  The optimal value, and thus a lower bound for the minimum potential energy of $m$ particles on a toric grid with $n$ nodes, is given by:
  \begin{equation}\label{EigenvalueBound}
    \lambda_1 \frac{m^2}{n}+\lambda_{\min}\left(m-\frac{m^2}{n}\right).
  \end{equation}
\end{prop}}

{\color{black}We will from now on refer to the bound \eqref{EigenvalueBound} as BDL-eigenvalue bound, named after the authors}.

{\color{black}
\begin{rem}
  The energy minimization problem on a toric grid is a special case of the so-called
   \emph{$m$-cluster} problem, or \emph{minimum weight $m$-subgraph} problem (see e.g. \cite{malick2012solving}).
   Indeed, the toric grid defines the graph in question, the edge weights are the energies between grid points,
   and the minimum weight $m$-subgraph corresponds to the $m$ grid points where the particles are placed in a minimum energy configuration.
   An SDP-bound with $n\times n$ matrix variables was proposed for the maximum weight $m$-subgraph
    problem  in \cite{malick2012solving}, which is of the form \ref{kClusterBound}:
  \begin{align}
  \inf\enspace & \frac{1}{4}e^TBe + \frac{1}{2}e^TBy +\frac{1}{4}\langle B,Y\rangle\label{kClusterBound}\\
    \mathrm{s.t. }\enspace& e^T y = 2m-n,\nonumber\\
    & \sum_{i}Y_{ij} = (2m-n)y_j\quad \text{for }j=1,\ldots,n,\nonumber\\
    & Y_{ii} = 1 \quad \text{for }i=1,\ldots,n,\nonumber\\
    & \begin{pmatrix}
        1 & y^T \\
        y & Y
      \end{pmatrix}\succcurlyeq 0.\nonumber
\end{align}

It is straightforward to check that this bound is at most as good as \eqref{FractTotalEnergy} for our problem. Indeed, since the weighted graph given by $B$ is vertex transitive, we obtain a feasible solution $(Y,y)$ for \eqref{kClusterBound} (with the same objective value) from a feasible solution $x$ for \eqref{FractTotalEnergy} by setting
$$ y = \mathrm{sym}(2x-e) = \left(\frac{2m}{n}-1\right)e \text{ and } Y = \mathrm{R}((2x-e)(2x-e)^T),$$
where $\mathrm{sym}$ averages a vector over the orbits of the automorphism group of the graph given by $B$, and $\mathrm{R}$ is the Reynolds operator
of the same group (i.e.\ it averages matrix entries over the $2$-orbits of the group). Thus this bound
is not strong enough to improve the BDL-eigenvalue bound \eqref{EigenvalueBound}.
\end{rem}

}

\subsection{Bound comparison}
We now want to compare the {\color{black}BDL-eigenvalue bound} with the QAP relaxations described last section, as well as the different QAP relaxations for this specific case.
\begin{prop}\label{EigenIsProjBound}
  The {\color{black}BDL-eigenvalue bound of \cite{bouman2013energy}, see \eqref{EigenvalueBound}}, coincides with
  the {\color{black}projected eigenvalue bound (see Proposition \ref{PBBound})} for the QAP problem \eqref{EnergyMinQAP}.
\end{prop}
\begin{proof}
  The matrix $D$ {\color{black}in Proposition \ref{PBBound}} is now given by
  \begin{equation*}
    D=\frac{2}{n}Aee^TB = \lambda_1\frac{2}{n} (\underbrace{m,\ldots,m}_{m \text{ times}},0,\ldots,0)^Te^T,
  \end{equation*}
 and thus has entries $\lambda_1\frac{2m}{n}$ in the first $m$ rows, and zeros  otherwise. As such, the permutation $\varphi$ does not influence the result, and
  \begin{equation}\label{LinearTermIsConstant}
    \min_\varphi\sum_{i=1}^{n}d_{i\varphi(i)} = \lambda_1\frac{2m^2}{n}.
  \end{equation}
  We know that $e$ is an eigenvector of $B$, hence we have
  \begin{equation*}
    \frac{(e^TAe)(e^TBe)}{n^2} = \frac{m^2\lambda_1n}{n^2} = \lambda_1 \frac{m^2}{n}.
  \end{equation*}
  The matrix $A$ has rank one, so $\tilde{A}$ has rank one as well. Since $e$ is an eigenvector of $B$, $\tilde{B}$ has the same eigenvalues as $B$, except for $\lambda_1$, thus we get
  \begin{equation*}
    \langle \lambda_{\tilde{A}},\mu_{\tilde{B}}\rangle^- = \lambda_{\max}(\tilde{A})\lambda_{\min}(\tilde{B}) = \mathrm{tr}(\tilde{A})\lambda_{\min}.
  \end{equation*}
  The eigenvalue of $\tilde{A}$ is exactly
  \begin{align*}
    \mathrm{tr}(\tilde{A}) &= \mathrm{tr}(V^TAV), \\
     &= \mathrm{tr}(A) - \frac{1}{n}\mathrm{tr}(AJ),\\
     &= m-\frac{m^2}{n}.
  \end{align*}
  Combining these, we see that the projection bound is the same as the eigenvalue bound:
  \begin{align*}
    PB(A,B) &= \langle \lambda_{\tilde{A}},\mu_{\tilde{B}}\rangle^- + \min_\varphi\sum_{i=1}^{n}d_{i\varphi(i)} - \frac{(e^TAe)(e^TBe)}{n^2} \\
    &= \lambda_1 \frac{m^2}{n}+\lambda_{\min}\left(m-\frac{m^2}{n}\right).
  \end{align*}
\end{proof}
Thus the BDL-eigenvalue  bound \eqref{EigenvalueBound} is the same as
the {\color{black}weakest} of the QAP bounds we considered, namely the bound $PB(A,B)$. Furthermore, even the convex quadratic bound
cannot give us better bounds here, as we show now.
\begin{prop}\label{PBisQPB}
  If $A$ and $B$ are of the form
  \begin{align*}
  {\color{black}a_{ij}}=\begin{cases}
           1, & \mbox{if } i,j\leq m \\
           0, & \mbox{otherwise}.
         \end{cases}, & \quad {\color{black}b_{i,j}} = f_{i,j} = f_{(x_i,y_i),(x_j,y_j)},
\end{align*}
as defined for the energy minimization problem, then we have that
\begin{equation*}
  PB(A,B)=QPB(A,B),
\end{equation*}
where
  \begin{align*}
       &PB(A,B) = \langle \lambda_{\tilde{A}},\mu_{\tilde{B}}\rangle^-+\min_\varphi\sum_{i=1}^{n}d_{i\varphi(i)}-\frac{(e^TAe)(e^TBe)}{n^2}\\
        & QPB(A,B) = \min_{\substack{X\geq 0 \text{ doubly stochastic}\\X=\frac{1}{n}ee^T + VYV^T\\y = \mathrm{vec}(Y)}} y^T\hat{Q}y +  \langle \lambda_{\tilde{A}},\mu_{\tilde{B}}\rangle^- + \frac{2}{n}\mathrm{tr}\left(BJAX\right) - \frac{(e^TAe)(e^TBe)}{n^2},
  \end{align*}
  where $D = (d_{ij}) = \frac{2}{n}AJB$ and $\hat{Q}$ is positive semidefinite, as defined before (see \eqref{PBBound}, \eqref{QPBBound}).
\end{prop}
\begin{proof}
  If we eliminate the terms which appear in both programs, we see that we want to show
  \begin{equation*}
   \min_\varphi\sum_{i=1}^{n}d_{i\varphi(i)}= \min_{\substack{X\geq 0 \text{ doubly stochastic}\\X=\frac{1}{n}ee^T + VYV^T\\y = \mathrm{vec}(Y)}} y^T\hat{Q}y  + \frac{2}{n}\mathrm{tr}\left(BJAX\right).
  \end{equation*}
  By definition of $D$ the two linear terms are equal, except that on the left we minimize over permutations, and on the right over doubly stochastic matrices. Because the terms are linear, and doubly stochastic matrices are convex combinations of permutations, the minima of the two linear terms are equal. Since $\hat{Q}$ is positive semidefinite, we thus want to find a doubly stochastic $X=\frac{1}{n}J+VYV^T$ with $y^T\hat{Q}y=0$, which minimizes the linear term. Earlier in \eqref{LinearTermIsConstant} we have seen that
  \begin{equation*}
    \sum_{i=1}^{n}d_{i\varphi(i)}  = \lambda_1\frac{2m^2}{n}\quad\forall \varphi,
  \end{equation*}
  and is thus the linear term is constant. Hence the term is also minimized for the average $X=\frac{1}{n}J$ of all permutations. For this $X$ we have $Y=\frac{1}{n}V^TJV=0$, and consequently $y^T\hat{Q}y=0$. Thus there is a feasible $X$ of $QPB(A,B)$ with objective value $PB(A,B)$, and the Proposition follows since $QPB(A,B) \geq PB(A,B)$.
\end{proof}

{\color{black}Thus it makes sense to consider the SDP-bound $SDPQAP(A,B)$ for the energy minimization QAP problem \eqref{ABDef}, if one wants to find stronger bounds than the BDL-bound used
in \cite{bouman2013energy}.}

\section{Reducing the relaxation of the energy minimization problem}
\label{sec:torusQAPreduce}
In this section {\color{black} we exploit the symmetry of the SDP-bound SDPQAP in the case of the energy minimization problem.
 Recall from Section \ref{EnergyMinSection}, and with reference to Figure \ref{GridToTorus}, that this is a
  quadratic assignment problem given by $A$ and $B$,} where
$B = (b_{ij})$ is indexed by toric grid points, and $b_{ij}$ $(i\neq j)$ equals the inverse of the Lee distance (shortest path on the grid) between
grid points $i$ and $j$, and $b_{ii} = 0$ for all $i$.
The matrix $A$ is zero except for a square block of all-ones in the upper left corner, of size equal to the number of particles on the toric grid.
 {\color{black} There are many approaches to exploit the symmetry of a conic optimization problem.
 Earlier work on group-theoretical symmetry reductions of SDP bounds for QAP was done in \cite{de2010exploiting,de2012improved}.
  We chose to apply the Jordan Reduction method of Parrilo and Permenter \cite{permenter2016dimension}  to our problem, which we will quickly summarize.}

\subsection{The Jordan Reduction}
{\color{black}
Parrilo and Permenter \cite{permenter2016dimension} introduced a set of three conditions a subspace has to fulfill,
such that it is be possible to use it for symmetry reduction. Here we  revisit some of their results.
\begin{defn}
  A \emph{projection} is a linear transformation $P\colon \mathcal{V}\to\mathcal{V}$  defined on a vector space $\mathcal{V}$ which is \emph{idempotent},
  i.e. $P^2=P$. If $\mathcal{V}$ is equipped with an inner product, the projection is called \emph{orthogonal},
  if it is self-adjoint with respect to this inner product.
   We denote the orthogonal projection onto a subspace
   $\mathcal{L} \subset \mathcal{V}$ by $P_{\mathcal{L}}$.
\end{defn}

We assume that the problem to be reduced is in the form
\begin{equation}\label{def:conic opt problem}
\left.
\begin{array}{llrl}
  \inf\enspace&\langle C,X\rangle  \\
  \mathrm{s.t.}\enspace& X\in X_0+\mathcal{L} \\
   & X\in S^n_+,
\end{array}
  \right\}
  \end{equation}
where $X_0 \in \mathbb{R}^{n\times n}$, and $\mathcal{L}$ is a linear subspace of $\mathbb{R}^{n\times n}$.

\begin{thm}[Theorem 5.2.4 and Proposition 1.4.1 in \cite{permenter2017reduction}]
\label{thm:equivalent CSICs}
Consider the conic optimization problem \eqref{def:conic opt problem} and let $S \subseteq \mathbb{R}^{n\times n}$ be a subspace of $\mathbb{R}^{n\times n}$. Define
$
C_{\mathcal{L}} = P_{\mathcal{L}}(C)
$
and $X_{0, \mathcal{L}^\perp} = P_{\mathcal{L}^\perp}(X_0)$.
If $S$ fulfills
 \begin{itemize}
      \item[(a)] $C_\mathcal{L},X_{0,\mathcal{L}^\perp} \in S$,
     \item[(b)] $P_\mathcal{L}(S) \subseteq S$,
      \item[(c)] $S\supseteq \{X^2 \enspace\colon\enspace X\in S\}$,
 \end{itemize}
then restricting the feasible set of conic program \eqref{def:conic opt problem}
to  $S$  results in another --- potentially significantly smaller --- program, with the same optimal value:
\begin{align*}
 \inf\enspace&\langle P_S(C),X\rangle \\
  \mathrm{s.t.}\enspace & X\in P_S(X_0)+\mathcal{L}\cap S, \\
  & X\in S^n_+\cap S.
\end{align*}

We call such an $S$ admissible for the problem \eqref{def:conic opt problem}.
\end{thm}
We will use the concept of a Jordan algebra: for our purposes this will be a subspace of symmetric
matrices that is closed under the product $X \circ Y = \frac{1}{2}(XY+YX)$. It follows from condition (c) in the theorem that
every admissible subspace is a Jordan algebra.
Indeed, a subspace of symmetric matrices is a Jordan algebra if and only if it is closed under taking squares, due to the identity
$
X \circ Y = \frac{1}{2}\left((X+Y)^2 - X^2 - Y^2\right).
$
We will denote the full Jordan algebra of symmetric $n \times n$ matrices by $\mathbb{S}^n$.

Note that, for the SDP-bound SDPQAP, one would actually use the cone of entry-wise nonnegative matrices (i.e.\ doubly nonnegative matrices), as opposed to $S^n_+$.
To deal with the nonnegativity, we will in fact work with admissible subspaces that have 0-1 bases and where the basis matrices have disjoint support (so-called partition
subspaces).
}

\subsection{Symmetric circulant matrices}
First, we need some well-known properties of (symmetric) circulant matrices, which will appear later in the construction of the admissible subspaces of the relaxation of the energy minimization problem.
\begin{defn}
  An $n\times n$ matrix $C$ is called \emph{circulant}, if each row is rotated one element to the right relative to the row above, i.e. $C_{ij}=c_{j-i\mod n}$ for all $i,j$ and fitting $c_k$, $k=0,\ldots, n-1$.
\end{defn}
\begin{prop}\label{SymmCirculantProp}
  A symmetric circulant $n\times n$ matrix $C$ has at most $\lfloor \frac{n}{2}\rfloor + 1= \lceil\frac{n+1}{2}\rceil$ unique entries, and $c_k = c_{n-k}$.
\end{prop}
\begin{proof}
  Let $j\geq i$ and $k = j-i$. By definition we have $c_k= C_{ij} = C_{ji} = c_{n-(j-i)}=c_{n-k}$. Hence $C$ is given by $c_0,\ldots, c_{\lfloor \frac{n}{2}\rfloor}$.
\end{proof}

This allows us to construct symmetric circulant matrices from a given $c\in\mathbb{R}^{\lceil\frac{n+1}{2}\rceil}$. We call this function $C=\mathrm{circ}_n(c_0,\ldots, c_{\lfloor \frac{n}{2}\rfloor})$.

\begin{prop}[E.g. Theorem 7 in \cite{gray2006toeplitz}]\label{SymCircProd}
  The product of two circulant matrices is a circulant matrix, and the product commutes. The product of symmetric circulant matrices is symmetric.
\end{prop}

We call the Jordan algebra  (with product $X \circ Y = \frac{1}{2}(XY+YX)=XY$) of symmetric circulant $n\times n$ matrices $\mathfrak{C}^n$.
 We define a $0/1$-basis for $\mathfrak{C}^n$ by
\begin{equation}\label{CirculantBasis}
  \left\lbrace C^n_i = \mathrm{circ}_n(d_i) \enspace\colon\enspace i = 0,\ldots, \lfloor \frac{n}{2}\rfloor\right\rbrace,
\end{equation}
where for $i\notin\{0, \frac{n}{2}\}$ we set $d_i = e_i\in \mathbb{R}^{\lfloor \frac{n}{2}\rfloor+1}$, the vector with a one in position $i$, and zero otherwise. For $i=0$ and $i=\frac{n}{2}$, if $n$ is even, we set $d_i = 2e_i$.

\subsection{Admissible subspaces}
{\color{black}
To  reduce the program $SDPQAP(A,B)$ in \eqref{QAPSDP} for the energy minimization problem \eqref{ABDef}, one should
 find an admissible subspace $S$ for every such problem. In this case $\mathcal{L}$ is the subspace given by the $Y\in {\mathbb{S}}^{n^2}$ with
\begin{align*}
\langle I_n\otimes E_{jj},Y\rangle &=0 \text{ for }j\in [n],\\
       \langle E_{jj}\otimes I_n,Y\rangle&=0 \text{ for }j\in [n],\\
       \langle T,Y\rangle &=0, \\
       \langle J_{n^2},Y\rangle &= 0,
\end{align*}
where $T=I_n\otimes (J_n-I_n)+(J_n-I_n)\otimes I_n$, and $X_0$ is any symmetric matrix satisfying the linear constraints of the SDP \eqref{QAPSDP}, e.g. $X_0 = \mathrm{vec}(I_n)\mathrm{vec}(I_n)^T$.

Note here that we are missing the constraint that $X$ is entry-wise nonnegative. But it is easy to check that if we restrict $S$ to be a \emph{partition} subspace (i.e. a subspace with an orthogonal $0/1$-basis), then $P_S(X)$ is doubly-nonnegative if $X$ is.
}

Recall that, for $n=n_1n_2$, the matrix $B\in \mathbb{R}^{n\times n}$ is defined
 by $B_{(x_1,y_1),(x_2,y_2)}=1/d_{\mathrm{Lee}}((x_1,y_1),(x_2,y_2))$, where $d_{\mathrm{Lee}}$ is the Lee-distance
  (length of shortest path on the toric grid). The ordering of the indices $[n_1]\times [n_2]=[n]$ we left implicit in {\color{black}earlier sections} of this paper,
   but now we fix it to $(x,y) \mapsto n_2(x-1)+y$. $A\in \mathbb{R}^{n\times n}$ is the matrix with an $m\times m$ all-one block in the top left corner,
   and otherwise zero.

In this section we will make use of Tensor products of algebras. As a reminder, if $A_1,\ldots, A_{d_1}\in \mathbb{R}^{n_1\times n_1}$
is a basis of a matrix algebra $\mathcal{A}$, and $B_1,\ldots, B_{d_2}\in \mathbb{R}^{n_2\times n_2}$ a basis of a matrix algebra $\mathcal{B}$,
 then $\mathcal{A}\otimes \mathcal{B}$ is the $n_1n_2\times n_1n_2$ matrix algebra with basis $A_i\otimes B_j$, for $i\in [d_1], j\in [d_2]$.

We restrict ourselves to a partition subspace, which means that the exact values of the entries of the matrix do not matter to us, only the pattern of unique elements. For the first of the three properties, we take a look at the structure of $\mathbf{C} \coloneqq B\otimes A$.


\begin{lem}\label{BStruct}
  $B\in \mathfrak{C}^{n_1}\otimes \mathfrak{C}^{n_2}$, i.e. $B$ is a block matrix, with $n_1$ rows and columns of blocks, which are arranged in a symmetric circulant pattern, and each of these blocks is an $n_2\times n_2$ symmetric circulant matrix.

\end{lem}
\begin{proof}
The Lee-distance between $(x_1,y_1)$ and $(x_2,y_2)$ depends only on $x_2-x_1\mod n_1$ and $y_2-y_1\mod n_2$, and the order of the arguments do not matter. This means that both the submatrices for fixed $x$ and for fixed $y$ coordinates are symmetric circulant matrices:
\begin{align*}
  \left(B_{(i,y_1),(j,y_2)}\right)_{1\leq i,j,\leq n_1} \in \mathfrak{C}^{n_1}, && \left(B_{(x_1,i),(x_2,j)}\right)_{1\leq i,j,\leq n_2} \in \mathfrak{C}^{n_2}.
\end{align*}
The chosen ordering of the indices $(x,y) \mapsto n_2(x-1)+y$ thus results in $B\in\mathfrak{C}^{n_1}\otimes \mathfrak{C}^{n_2}$.
\end{proof}

In the case $n_1 = n_2$ we can restrict the algebra further.
\begin{lem}\label{BStructSquare}
  If $n_1 = n_2$, then
  \begin{equation}\label{CnnDef}
    B\in \mathfrak{C}^{n_1,n_1}\coloneqq\left\lbrace X\in \mathfrak{C}^{n_1}\otimes \mathfrak{C}^{n_1}\enspace \colon \enspace X_{(x_1,y_1),(x_2,y_2)} = X_{(y_1,x_1),(y_2,x_2)}\right\rbrace,
  \end{equation}
  and $\mathfrak{C}^{n_1,n_1}$ is a Jordan sub-algebra of $\mathfrak{C}^{n_1}\otimes \mathfrak{C}^{n_1}$.
\end{lem}
\begin{proof}
  $B$ is has this symmetry by definition of the Lee-distance. $\mathfrak{C}^{n_1,n_1}$ is a sub-algebra, because it is the restriction of an algebra to the commutant of $\{P,I\}$, where $P$ is the $n\times n$ permutation matrix switching the indices corresponding to each $(x,y)$ with the one corresponding to $(y,x)$.
\end{proof}

The other relevant Jordan algebra for our problem is described in the following proposition.
\begin{prop}\label{JnmProp}
  The subspace of $n\times n$ matrices with pattern
  \begin{center}
\begin{tikzpicture}
        \matrix [matrix of math nodes,left delimiter=(,right delimiter=)] (m)
        {
            a       & b         & \cdots    & b         & c         & \cdots    & \cdots    & c \\
            b       & \ddots    & \ddots    & \vdots    & \vdots    &           &           & \vdots \\
            \vdots  & \ddots    & \ddots    & b         & \vdots    &           &           & \vdots \\
            b       & \cdots    & b         & a         & c         & \cdots    & \cdots    & c\\
            c       & \cdots    & \cdots    & c         & d         & e         & \cdots    & e \\
            \vdots  &           &           & \vdots    & e         & \ddots    & \ddots    & \vdots \\
            \vdots  &           &           & \vdots    & \vdots    & \ddots    & \ddots    & e \\
            c       & \cdots    & \cdots    & c         & e         & \cdots    & e         & d \\
        };
        \draw[color=black, thick]  ($(m-4-1.south west)!0.5!(m-5-1.north west)$) -- ($(m-4-8.south east)!0.5!(m-5-8.north east)$);
        \draw[color=black, thick]  ($(m-8-4.south east)!0.5!(m-8-5.south west)$) -- ($(m-1-4.north east)!0.5!(m-1-5.north west)$);
        \draw[decorate,decoration = {brace}, transform canvas={xshift=-3em},thick] (m-4-1.south) -- (m-1-1.north) node[left=5pt,midway] {$m$};
        \draw[decorate,decoration = {brace}, transform canvas={xshift=-3em},thick] (m-8-1.south) -- (m-5-1.north) node[left=5pt,midway] {$n-m$};

\end{tikzpicture}
\end{center}
form a Jordan algebra, say $\mathcal{J}^{n,m}$. We call the $0/1$-basis corresponding to this pattern $J_A,J_B,J_C,J_D,J_E$.
\end{prop}
\begin{proof}
  A straightforward calculation shows that squaring such a matrix results in another matrix of the same pattern with parameters
  \begin{align*}
  a' &= a^2+(m-1)b^2 + (n-m) c^2, \\
  b' &= 2ab + (m-2)b^2 + (n-m) c^2,\\
  c' &= (a+(m-1)b)c + (d + (n-m-1)e)c,\\
  d' &= d^2 + (n-m-1)e^2 + m c^2,\\
  e' &= 2de + (n-m-2) e^2 + m c^2.
\end{align*}
\end{proof}

We now want to show that the space $S \coloneqq \mathfrak{C}^{n_1}\otimes \mathfrak{C}^{n_2}\otimes \mathcal{J}^{n,m}$,
 respectively $S = \mathfrak{C}^{n_1,n_1}\otimes \mathcal{J}^{n,m}$ if $n_1=n_2$, is admissible. We do this by verifying the three
 conditions listed in Theorem \ref{thm:equivalent CSICs}.

 \begin{thm}
  The subspace $S \coloneqq \mathfrak{C}^{n_1}\otimes \mathfrak{C}^{n_2}\otimes \mathcal{J}^{n,m}$, respectively $S = \mathfrak{C}^{n_1,n_1}\otimes \mathcal{J}^{n,m}$ if $n_1=n_2$ is admissible for \eqref{QAPSDP}, where $B$ and $A$ are the matrices corresponding to the problem of minimizing the energy of $m$ particles on an $n_1\times n_2$ grid.

  For $2<m<n-2$ the dimension of $S$ is $5\lceil\frac{n_1+1}{2}\rceil \lceil\frac{n_2+1}{2}\rceil$ in the case $n_1\neq n_2$, and $\frac{5}{2}\lceil\frac{n_1+1}{2}\rceil \left(\lceil\frac{n_2+1}{2}\rceil+1\right)$ in the case $n_1=n_2$.
\end{thm}
\begin{proof}
We first show that $P_\mathcal{L}(S)\subseteq S$.
 To this end, note that both $T=I_n\otimes (J_n-I_n)+(J_n-I_n)\otimes I_n$ and $J_{n^2}$ are elements of $S$, since
  $J_n = J_{n_1}\otimes J_{n_2}$ and $I_n = I_{n_1}\otimes I_{n_2}$ are both in $\mathfrak{C}^{n_1}\otimes \mathfrak{C}^{n_2}$ (and in $\mathfrak{C}^{n_1,n_1}$ if $n_1=n_2$),
  as well as in $\mathcal{J}^{n,m}$ because $I_n = J_A+J_D$ and $J_n = I_n + J_B+J_C+J_E$. Thus
  $T$ and $J_{n^2}$ can be written as linear combination of Kronecker products of elements
  of $\mathfrak{C}^{n_1}$, $\mathfrak{C}^{n_2}$ and $\mathcal{J}^{n,m}$, and are as such elements of $S$.

  The other two constraints are given by matrices $I_n\otimes E_{jj}$ and $E_{jj}\otimes I_n$, which only overlap with the two basis elements $C^{n_1}_0\otimes C^{n_2}_0\otimes J_A$ and $C^{n_1}_0\otimes C^{n_2}_0\otimes J_D$. Since $C^{n_1}_0\otimes C^{n_2}_0\otimes J_A = \sum_{j=1}^{m}I_n\otimes E_{jj}$ and  $C^{n_1}_0\otimes C^{n_2}_0\otimes J_D = \sum_{j=m+1}^{n}I_n\otimes E_{jj}$, both of these matrices are projected to zero.

  Thus all basis elements of $S$ are sent to elements of $S$, and $P_\mathcal{L}(S)\subseteq S$.

  Next, we show $\mathbf{C}_\mathcal{L} \in S$.
  By Lemma \ref{BStruct}, Lemma \ref{BStructSquare} and the definition of $A$, we know that $C=B\otimes A\in S$. Since $P_\mathcal{L}(S)\subseteq S$, that $\mathbf{C}_\mathcal{L} \in S$ as well.

  Next, we show that $X_{0,\mathcal{L}^\perp}\in S$. {\color{black}To project $X_0=\mathrm{vec}(I_n)\mathrm{vec}(I_n)^T$ onto $\mathcal{L}^\perp$, the span of the constraint
       matrices, we first notice only two of them have nonzero entries outside of the diagonal,
       the all one matrix $J_{n^2}$, and the matrix  $I_n\otimes (J_n-I_n)+(J_n-I_n)\otimes I_n$,
       which we will call $T$ from now on. The matrices $I_n\otimes E_{jj}$ for $j=1,\ldots,n$ sum to the identity matrix $I_{n^2}$,
        which means that we can easily find an orthogonal basis of the off diagonal part of $\mathcal{L}^\perp$:
     \begin{align*}
       B_1 &= T, \\
       B_2 &= J_{n^2} - I_{n^2} - T.
     \end{align*}
     Since $\langle T,X_0\rangle =0$ and $\langle J_{n^2},X_0\rangle = n^2$ we get
     \begin{equation*}
       \langle B_2, X_0\rangle = \langle J_{n^2},X_0\rangle - \langle I_{n^2},X_0\rangle = n^2 - n.
     \end{equation*}
     Hence the off-diagonal part of $X_{0,\mathcal{L}^\perp}$ is the matrix
     \begin{equation*}
       \frac{\langle B_2,X_0\rangle}{\langle B_2,B_2\rangle} B_2 = \frac{n^2-n}{n^4-n^2 - 2n(n^2-n)}B_2=\frac{1}{n^2-n}B_2.
     \end{equation*}
     The diagonal part of $X_{0,\mathcal{L}^\perp}$ is the matrix $\frac{1}{n} I_{n^2}$, since
     \begin{align*}
       \langle E_{jj}\otimes I_n, \frac{1}{n} I_{n^2} - X_0\rangle &= \langle E_{jj}\otimes I_n, \frac{1}{n} I_{n^2}\rangle -\langle E_{jj}\otimes I_n, X_0\rangle = 1-1=0,
     \end{align*}
     and analogously $\langle I_n\otimes E_{jj} , \frac{1}{n} I_{n^2} - X_0\rangle = 0$. Combining the two parts we see
     \begin{equation*}
       X_{0,\mathcal{L}^\perp} = \frac{1}{n^2-n}B_2 + \frac{1}{n} I_{n^2}.
     \end{equation*}
  }
   Since $J_{n^2}$, $I_{n^2}$ and $T$ are elements of $S$, we get that $X_{0,\mathcal{L}^\perp}\in S$.

  Finally, we note that $S$ is a Jordan algebra. This completes the proof that $S$ is admissible. The dimension of $S$ follows from $\mathfrak{C}^n$ having dimension $\lceil\frac{n+1}{2}\rceil$
  and $\mathcal{J}^{n,m}$ having dimension $5$. In the case $n_1=n_2$ the dimension is lower, since we can combine the basis elements $C^{n_1}_i\otimes C^{n_1}_j$ and $C^{n_1}_j\otimes C^{n_1}_i$ for each pair $i\neq j$.
\end{proof}

Thus we have found an admissible subspace $S$ for \eqref{QAPSDP}, where $A$ and $B$ are the matrices corresponding to the problem of minimizing the energy of $m$ particles on an $n_1\times n_2$ toric grid. Its dimension is of order $\mathcal{O}(n_1n_2)$, which is significantly
 less than the original number of variables $\frac{n_1^4n_2^4+n_1^2n_2^2}{2}=\mathcal{O}(n_1^4n_2^4)$.
 The number of variables can be reduced further by fixing the variables corresponding to nonzero entries of $I_n\otimes (J_n-I_n)+(J_n-I_n)\otimes I_n$ to zero.
 Thus if $\{B_1,\ldots,B_k\}$ is a $0/1$-basis of an admissible subspace, then it is enough to optimize over variables in the subspace $S_0$ with basis
\begin{equation*}
  \left\lbrace B_i\enspace\colon\enspace \langle B_i,I_n\otimes (J_n-I_n)+(J_n-I_n)\otimes I_n\rangle = 0 \right\rbrace.
\end{equation*}

This results in $3\lceil\frac{n_1+1}{2}\rceil \lceil\frac{n_2+1}{2}\rceil-1$ variables in the case $n_1\neq n_2$, and $1.5\lceil\frac{n_1+1}{2}\rceil \left(\lceil\frac{n_2+1}{2}\rceil+1\right)-1$ variables in the case $n_1=n_2$. A few examples can be seen in Table \ref{ReducedDimTable}. Note that the resulting subspace is generally not a Jordan algebra anymore.

\begin{table}[htbp] \centering
 \caption{Number of variables before and after symmetry reduction.}
 \label{DimensionTable}
\begin{tabular}{ccHcc}
\toprule
$(n_1,n_2)$&$\mathrm{dim}({\mathbb{S}}^{n_1^2n_2^2})$&Before (no zeroes)&$\mathrm{dim}(S)$&$\mathrm{dim}(S_0)$\\
\midrule
$(4, 4)$& 32896& 29056& 30& 17\\
$(5, 5)$& 195625& 180625& 30& 17\\
$(6, 6)$& 840456& 795096& 50& 29\\
$(8, 8)$& 8390656&8132608& 75& 44\\
$(10, 10)$& 50005000&49015000& 105& 62\\
$(12, 12)$& 215001216& 212035968& 140& 83\\
$(24, 24)$& 55037822976& 54847051776& 455& 272\\
$(100, 100)$& $\approx 5\cdot 10^{15}$& $\approx 5\cdot 10^{15}$& 6630& 3977\\
$(1000, 1000)$& $\approx 5\cdot 10^{23}$ & $\approx 5\cdot 10^{23}$ & 628755& 377252\\
\midrule
$(6, 5)$& 405450& 379350& 60& 35\\
$(10, 5)$& 3126250& 3003750& 90& 53\\
$(24, 12)$& 3439895040& 3416090112& 455& 272\\
\bottomrule
\end{tabular}
\label{ReducedDimTable}
\end{table}

\subsection{Block diagonalization}
\label{sec:Block diagonalization}
We now want to block diagonalize the admissible subspace $S \coloneqq \mathfrak{C}^{n_1}\otimes \mathfrak{C}^{n_2}\otimes \mathcal{J}^{n,m}$, respectively
$S = \mathfrak{C}^{n_1,n_1}\otimes \mathcal{J}^{n,m}$ if $n_1=n_2$. We do this by making use of the fact that $S$ is a tensor product of algebras, which allows us to block diagonalize each part on its own.

\begin{lem}[See, for example, \cite{gray2006toeplitz}, \cite{de2008semidefinite}]
  The $0/1$-basis $\{C_i^n\enspace\colon\enspace i=0,\ldots, \lfloor \frac{n}{2}\rfloor\}$ of $\mathfrak{C}^n$ has a common set of eigenvectors, given by the columns of the discrete Fourier transform matrix:
  \begin{equation*}
    Q^n_{ij}\coloneqq \frac{1}{\sqrt{n}}e^{-2\pi\sqrt{-1}ij/n},\qquad i,j=0,\ldots, n-1.
  \end{equation*}
  The eigenvalues are
  \begin{equation*}
    \lambda_m(C_k^n) = 2\cos(2\pi mk/n),\qquad m=0,\ldots, n-1, k=0,\ldots, \lfloor \frac{n}{2}\rfloor,
  \end{equation*}
  and note that
  \begin{equation*}
    \lambda_m(C_k^n) = \lambda_{n-m}(C_k^n), \qquad m=1,\ldots,\lfloor \frac{n}{2}\rfloor, k=0,\ldots, \lfloor \frac{n}{2}\rfloor.
  \end{equation*}
\end{lem}
Thus we can block diagonalize $\mathfrak{C}^n$ by sending $C_k^n$ to the vector
\begin{equation*}
  \hat{\lambda}(C_k^n)\coloneqq (\lambda_0(C_k^n),\ldots, \lambda_{\lfloor \frac{n}{2}\rfloor}(C_k^n)).
\end{equation*}

To block diagonalize $\mathcal{J}^{n,m}$, one may use the Jordan isomorphism $\phi:\mathcal{J}^{n,m} \rightarrow \mathbb{R} \oplus \mathbb{R} \oplus \mathbb{S}^2$ given by
\[
\def\matr#1{
\left(
\begin{smallmatrix}
#1
\end{smallmatrix}
\right)}
\phi(J_A) = \matr{
n-m\\
&0\\
&&1&0\\
&&0&0\\
},
\; \phi(J_B) = \matr{
-1\\
&0\\
&&m-1&0\\
&&0&0\\
},
\; \phi(J_C) = \sqrt{m(n-m)} \matr{
0\\
&0\\
&&0&1\\
&&1&0\\
},
\]
\[
\def\matr#1{
\left(
\begin{smallmatrix}
#1
\end{smallmatrix}
\right)}
\phi(J_D) = \matr{
0\\
&1\\
&&0&0\\
&&0&1\\
},
\; \phi(J_E) = \matr{
0\\
&-1\\
&&0&0\\
&&0&n-m-1\\
}.
\]
This isomorphism was used implicitly in \cite{KOP}, but may also be verified directly by
confirming that  $\phi(X^2) = [\phi(X)]^2$ for all $X \in \mathcal{J}^{n,m}$.

We can now combine these block diagonalizations by noticing that it is enough to block diagonalize each of the algebras separately; see, for example, Section 7.2. in \cite{de2010exploiting}.
We obtain the final reduction shown in the next theorem. The proof is omitted since it is straightforward: {\color{black} One just has to calculate the inner products between the basis elements of the algebra and the data matrices, and eliminate variables fixed to zero by $\langle B_i,I_n\otimes (J_n-I_n)+(J_n-I_n)\otimes I_n\rangle = 0$. We then further scaled some variables and matrices to simplify terms further.

In the following we use the constants
\begin{equation*}
 d^{kl}_{ij}\coloneqq \cos\left(\frac{2\pi ki}{n_1}\right)\cos\left(\frac{2\pi lj}{n_2}\right),
\end{equation*}
which arise from the  diagonalization of the circulant matrices.
}

\begin{thm}\label{ReducedSDP}

The bound from \eqref{QAPSDP}, where the matrices $A,B$ correspond to the energy minimization problem with parameters $n_1$, $n_2$, $n=n_1n_2$ and $m$, equals the optimal value of the following semidefinite program:

\begin{align}
      \inf\enspace & n\sum_{i+j>0}\frac{y^{b \to b}_{ij}}{i+j}\label{ReducedSDPPrimal}\\
       \mathrm{s.t.}\enspace 
       &\sum_{i+j>0}\left(y^{b\to b}_{ij}+y^{w\to w}_{ij} + y^{b\leftrightarrow w}_{ij}\right)=n-1,\label{RedSDPSum}\\
       &\text{for all } 0\leq k\leq \left\lfloor\frac{n_1}{2}\right\rfloor, 0\leq l\leq \left\lfloor\frac{n_2}{2}\right\rfloor\colon\\
       &\quad \begin{pmatrix}
               \frac{m}{n} & 0 \\
               0 & \frac{n-m}{n}
             \end{pmatrix} + \sum_{i+j>0}d^{kl}_{ij}\begin{pmatrix}
               y^{b \to b}_{ij} & \frac{1}{2} y^{b\leftrightarrow w}_{ij} \\
               \frac{1}{2}y^{b\leftrightarrow w}_{ij} & y^{w \to w}_{ij}
             \end{pmatrix}\succcurlyeq 0,\label{RedSDPPSD}\\
       &\quad \frac{m(m-1)}{n}-\sum_{i+j>0}d^{kl}_{ij}y^{b \to b}_{ij}\geq 0,\label{RedSDPB}\\
       &\quad \frac{(n-m)(n-m-1)}{n}-\sum_{i+j>0}d^{kl}_{ij}y^{w \to w}_{ij}\geq 0\label{RedSDPW},\\
       &y_{ij}^{b\to b},y_{ij}^{w\to w},y_{ij}^{b\leftrightarrow w}  \ge 0 \; \forall (i,j) \enspace | \enspace i+j>0, \enspace i\in\left\{0,\ldots,\left\lfloor\frac{n_1}{2}\right\rfloor\right\}, \enspace j\in\left\{0,\ldots,\left\lfloor\frac{n_2}{2}\right\rfloor\right\}.
\end{align}
\end{thm}

{\color{black} We can interpret the variables as averaged occurrences of pairs of black points ($b \to b$), pairs of white points ($w\to w$),
and pairs of a white and a black point ($b\to w$), at distance $(i,j)$.
I.e.\ if we have a given configuration, we can count how many pairs of points at distance $(i,j)$ are both black,
 and then set $y^{b\to b}_{ij}$ to this value divided by the total number of pairs, to construct a feasible solution.
  Variables corresponding to distances bigger than $\left\lfloor\frac{n_1}{2}\right\rfloor$ respectively
   $\left\lfloor\frac{n_2}{2}\right\rfloor$ are looped around and added onto smaller distance variables,
   i.e.\ we may have variable values bigger than one.}

{\color{black}
 The semidefinite program in Theorem \ref{ReducedSDP} has block sizes of order at most $2\times 2$,  and is therefore
 a second-order cone program, which can be solved very efficiently; see e.g.\ \cite{Lobo_SOCP}.
Thus we were able to solve the SDP relaxation for toric grids
of sizes up to $100 \times 100$. Subsequently we were also able to prove optimality of certain configurations of particles
on toric grids, as detailed in the next section.}

\section{Numerical results}
\label{sec:numerical results}

Here we compare the eigenvalue-bound with the SDP-bound for the energy minimization problems. The upper bounds were found using simulated annealing {\color{black} (see
Algorithm \ref{alg:SimAnnAlg}).

\begin{algorithm}[H]
  \caption{Simulated annealing algorithm}
  \label{alg:SimAnnAlg}
    $iter \gets$ number of iterations to perform \\
    $P \gets $ random configuration of $m$ particles\\
    $val \gets $ energy of P $ = E(P)$\\
    $T \gets 1$\\
    $a \gets \sqrt[iter]{\frac{1}{iter}}$\\
    \For{$it\leftarrow 1$ \KwTo iterations}{
        $T \gets aT$\\
        $P' \gets P$\\
        Move a random particle of $P'$ to a neighbouring position\\

        \If{$E(P') < val$ or $\exp(-(E(P') - val) / T)\geq \mathrm{rand}(0,1)$}{
            $val \gets E(P')$\\
            $P\gets P'$
        }
    }
    \Return{$val$}
\end{algorithm}
}

Calculating the SDP-bound directly is prohibitively slow, which is why we symmetry reduced the problems first, as described in {\color{black} Section \ref{sec:torusQAPreduce}.
 After this reduction we can calculate these bounds very efficiently,
 solving a case of the problem on a $10\times 10$-grid in 0.2s, 1.3s on a $50\times 50$-grid, and in about 40s on $100\times 100$-grid, using Mosek on a 4-core
  3.4GHz Processor.}
In \cite{bouman2013energy} Bouman, Draisma and van Leeuwaarden prove optimality for the checkerboard arrangement in the cases that $n_1=n_2$ are even, and $m=\frac{n_1n_2}{2}$. This can be seen for the grid sizes we checked as well, but we do get some more proofs of optimality.

If one of the bounds is sharp, then we get a proof of optimality for these parameters. Furthermore, we can even prove optimality in some cases even if the bound is not completely sharp, as explained in the following result.
\begin{prop}\label{RoundErrorNotImportant}
  Let $V=\{v_1,\ldots,v_k\}$ be the set of unique entries of the matrix of potentials $B$. The one-dimensional shortest-vector-problem for these values is
  \begin{align*}
    r = \min \enspace &\vert s-t\vert \\
    \mathrm{s.t. } \enspace& s,t \in \left\lbrace x = 2\sum_{i=1}^{k}\alpha_i v_i \enspace\vert\enspace \alpha_i \in \mathbb{Z} \right\rbrace\\
    & s\neq t.
  \end{align*}
  If a feasible solution of the QAP has objective value that differs by less than $r$ from the SDP lower bound, then this feasible solution is optimal.
\end{prop}
\begin{proof}
  Both the matrix $A$ and the optimization variable $X\in \Pi_n$ of this QAP are symmetric $0/1$-matrices, and $B$ is symmetric as well, with zeros on the main diagonal, which means that the objective value is of the form $\sum_{i,j,k,l=1}^{n} a_{ik} b_{jl} X_{ij}X_{kl} = \sum_{i,j=1, i<j}^{n} 2 \alpha_{ij} b_{ij}$, where $\alpha_{ij}\in\mathbb{Z}$. Hence different objective values have to at least differ by $r$.
\end{proof}

For the inverse Lee-distance potential on a $6\times 6$-grid the set $V$ is $\{0,1,\frac{1}{2},\frac{1}{3},\frac{1}{4},\frac{1}{5},\frac{1}{6}\}$,
and $r = 2\left(\frac{1}{6}+\frac{3}{5}-\frac{3}{4}\right)=\frac{1}{30}$, which is optimal since 30 is the least common
denominator of the fractions in $2V$. Similarly one finds $r=\frac{1}{210}$ for a $7\times 7$-grid, $r=\frac{1}{420}$
 for a $8\times 8$-grid and $r=\frac{1}{1260}$ on a $10\times 10$-grid.
 In the Tables \ref{6table},\ref{7table}, \ref{8table} and \ref{10table} we give
 the bounds for square grids of sizes $6$,$7$,$8$ and $10$ respectively. As proven in
 Proposition \ref{InvertArrangement}, we only need to consider $m\leq \frac{n}{2}$. Bold
  font in the tables signify sharp bounds, in the sense of Proposition \ref{RoundErrorNotImportant}, which
   we then illustrate in Figures \ref{6figure}, \ref{7figure}, \ref{8figure} and \ref{10figure}.

\begin{table}[htbp] \centering
 \caption{The bounds on a 6x6 grid}
 \label{6table}
\vspace{0.2em}
 \begin{scriptsize}
\begin{tabular}{cccc}
\toprule
$m$&$PB(A,B)$&$SDPQAP(A,B)$&\begin{tabular}{@{}c@{}}Upper bounds from \\  simulated annealing\end{tabular}\\
\midrule
1&-1.514815&\textbf{0.000000}&\textbf{0.000000}\\
2&-2.125926&\textbf{0.333319}&\textbf{0.333333}\\
3&-1.833333&1.349939&1.500000\\
4&-0.637037&\textbf{2.999892}&\textbf{3.000000}\\
5&1.462963&5.416640&5.666667\\
6&4.466667&8.599983&8.666667\\
7&8.374074&12.622685&13.000000\\
8&13.185185&17.407305&17.600000\\
9&18.900000&22.937178&23.466667\\
10&25.518519&29.212957&29.666667\\
11&33.040741&36.233780&36.666667\\
12&41.466667&\textbf{44.000000}&\textbf{44.000000}\\
13&50.796296&53.065277&54.366667\\
14&61.029630&62.959998&64.666667\\
15&72.166667&73.687450&75.500000\\
16&84.207407&85.273263&86.666667\\
17&97.151852&97.718432&98.666667\\
18&\textbf{111.000000}&\textbf{111.000000}&\textbf{111.000000}\\
\bottomrule
\end{tabular}
\end{scriptsize}
\end{table}

\begin{table}[htbp] \centering
 \caption{The bounds on a 7x7 grid}
 \label{7table}
\vspace{0.2em}
 \begin{scriptsize}
\begin{tabular}{cccc}
\toprule
$m$&$PB(A,B)$&$SDPQAP(A,B)$&\begin{tabular}{@{}c@{}}Upper bounds from \\  simulated annealing\end{tabular}\\
\midrule
1&-1.535637&\textbf{0.000000}&\textbf{0.000000}\\
2&-2.287844&\textbf{0.333330}&\textbf{0.333333}\\
3&-2.256623&1.243763&1.300000\\
4&-1.441972&2.723982&2.800000\\
5&0.156109&4.784851&4.866667\\
6&2.537619&7.533726&7.800000\\
7&5.702558&10.916369&10.966667\\
8&9.650926&15.043550&15.500000\\
9&14.382724&19.814560&20.366667\\
10&19.897950&25.325560&25.900000\\
11&26.196607&31.554779&32.166667\\
12&33.278692&38.455887&39.033333\\
13&41.144207&46.029212&46.733333\\
14&49.793151&54.274568&54.933333\\
15&59.225525&63.260772&64.433333\\
16&69.441328&73.172931&74.100000\\
17&80.440560&83.797024&85.200000\\
18&92.223221&95.162225&96.600000\\
19&104.789312&107.298752&109.033333\\
20&118.138832&120.154716&122.000000\\
21&132.271782&133.732016&134.866667\\
22&147.188160&148.029982&150.066667\\
23&162.887969&163.048746&165.700000\\
24&179.371206&179.371185&182.266667\\
\bottomrule
\end{tabular}
\end{scriptsize}
\end{table}

\begin{table}[htbp] \centering
 \caption{The bounds on a 8x8 grid}
 \label{8table}
\vspace{0.2em}
\begin{scriptsize}
 \begin{tabular}{cccc}
\toprule
$m$&$PB(A,B)$&$SDPQAP(A,B)$&\begin{tabular}{@{}c@{}}Upper bounds from \\  simulated annealing\end{tabular}\\
\midrule
1&-1.670238&\textbf{0.000000}&\textbf{0.000000}\\
2&-2.654762&\textbf{0.249994}&\textbf{0.250000}\\
3&-2.953571&1.014038&1.133333\\
4&-2.566667&\textbf{2.266433}&\textbf{2.266667}\\
5&-1.494048&4.062460&4.233333\\
6&0.264286&6.435304&6.583333\\
7&2.708333&9.423375&9.666667\\
8&5.838095&12.965443&13.000000\\
9&9.653571&17.078833&17.442857\\
10&14.154762&21.749475&22.126190\\
11&19.341667&26.990007&27.628571\\
12&25.214286&32.848535&33.666667\\
13&31.772619&39.445606&40.352381\\
14&39.016667&46.636662&47.350000\\
15&46.946429&54.421402&54.950000\\
16&55.561905&62.799758&63.076190\\
17&64.863095&71.971752&72.921429\\
18&74.850000&81.768179&83.023810\\
19&85.522619&92.238774&93.535714\\
20&96.880952&103.368587&104.300000\\
21&108.925000&115.126506&116.114286\\
22&121.654762&127.522322&128.483333\\
23&135.070238&140.541217&141.719048\\
24&149.171429&154.193846&155.514286\\
25&163.958333&168.487184&170.390476\\
26&179.430952&183.448522&185.711905\\
27&195.589286&199.055388&201.416667\\
28&212.433333&215.278915&217.333333\\
29&229.963095&232.135382&234.083333\\
30&248.178571&249.661718&251.083333\\
31&267.079762&267.846258&268.750000\\
32&286.666667&\textbf{286.666665}&\textbf{286.666667}\\
\bottomrule
\end{tabular}
\end{scriptsize}
\end{table}

\begin{table}[htbp] \centering
 \caption{The bounds on a 10x10 grid}
 \label{10table}
\vspace{0.2em}
\begin{scriptsize}
 \begin{tabular}{cccc}
\toprule
$m$&$PB(A,B)$&$SDPQAP(A,B)$&\begin{tabular}{@{}c@{}}Upper bounds from \\  simulated annealing\end{tabular}\\
\midrule
1&-1.716889&\textbf{0.000000}&\textbf{0.000000}\\
2&-2.883429&\textbf{0.199988}&\textbf{0.200000}\\
3&-3.499619&0.811542&0.904762\\
4&-3.565460&\textbf{1.807631}&\textbf{1.809524}\\
5&-3.080952&3.233819&3.333333\\
6&-2.046095&5.131049&5.233333\\
7&-0.460889&7.479826&7.700000\\
8&1.674667&10.307204&10.355556\\
9&4.360571&13.579004&13.800000\\
10&7.596825&17.298929&17.433333\\
11&11.383429&21.467188&21.876190\\
12&15.720381&26.234485&26.679365\\
13&20.607683&31.472262&32.029365\\
14&26.045333&37.183990&37.744444\\
15&32.033333&43.378553&44.096825\\
16&38.571683&50.037058&50.736508\\
17&45.660381&57.183049&57.922222\\
18&53.299429&64.786226&65.342857\\
19&61.488825&72.867194&73.285714\\
20&70.228571&\textbf{81.428359}&\textbf{81.428571}\\
21&79.518667&90.810097&91.739683\\
22&89.359111&100.699416&102.068254\\
23&99.749905&111.093201&112.648413\\
24&110.691048&121.990693&123.492857\\
25&122.182540&133.400133&134.938889\\
26&134.224381&145.304277&146.824603\\
27&146.816571&157.729721&159.401587\\
28&159.959111&170.652552&172.295238\\
29&173.652000&184.080092&185.607937\\
30&187.895238&198.017798&199.388095\\
31&202.688825&212.458779&214.014286\\
32&218.032762&227.417991&228.861111\\
33&233.927048&242.967654&244.254762\\
34&250.371683&259.050406&260.466667\\
35&267.366667&275.649622&277.788889\\
36&284.912000&292.759839&295.125397\\
37&303.007683&310.386054&312.823810\\
38&321.653714&328.520878&331.160317\\
39&340.850095&347.242354&349.413492\\
40&360.596825&366.466726&368.701587\\
41&380.893905&386.186323&389.195238\\
42&401.741333&406.463567&410.162698\\
43&423.139111&427.236004&431.381746\\
44&445.087238&448.524441&452.888889\\
45&467.585714&470.378907&474.000000\\
46&490.634540&492.906204&496.033333\\
47&514.233714&515.926534&518.650000\\
48&538.383238&539.531651&541.466667\\
49&563.083111&563.667331&564.800000\\
50&\textbf{588.333333}&\textbf{588.332003}&\textbf{588.333333}\\
\bottomrule
\end{tabular}
\end{scriptsize}
\end{table}


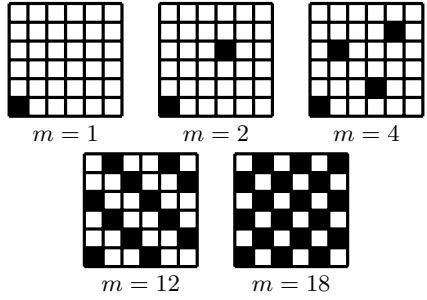
\begin{figure}
\begin{center}
\begin{tikzpicture}[scale = 0.5]
\begin{scope}[local bounding box=s661, xshift = -4cm, scale = 0.5]

   \draw [step=1.0,black, very thick] (1,1) grid (7,7);

   \foreach \x/\y in {1/1}
      \fill[draw=black] (\x,\y) rectangle ++(1,1);
    \node[below,font=\small] at (s661.south) {$m=1$};
\end{scope}
\begin{scope}[local bounding box=s662, xshift = 0cm, scale = 0.5]
\draw [step=1.0,black, very thick] (1,1) grid (7,7);

   \foreach \x/\y in {1/1,4/4}
      \fill[draw=black] (\x,\y) rectangle ++(1,1);
   \node[below,font=\small] at (s662.south) {$m=2$};
\end{scope}
\begin{scope}[local bounding box=s664, xshift = 4cm, scale = 0.5]
   \draw [step=1.0,black, very thick] (1,1) grid (7,7);

   \foreach \x/\y in {1/1, 2/4, 4/2, 5/5}
      \fill[draw=black] (\x,\y) rectangle ++(1,1);
   \node[below,font=\small] at (s664.south) {$m=4$};
\end{scope}
\begin{scope}[local bounding box=s6612, xshift = -2cm, yshift = -4cm, scale = 0.5]
   \draw [step=1.0,black, very thick] (1,1) grid (7,7);

   \foreach \x/\y in {1/1, 1/4, 2/3,3/2,4/1, 2/6,6/2, 5/3,3/5,4/4, 6/5,5/6}
      \fill[draw=black] (\x,\y) rectangle ++(1,1);
   \node[below,font=\small] at (s6612.south) {$m=12$};
\end{scope}
\begin{scope}[local bounding box=s6618, xshift = 2cm, yshift = -4cm, scale = 0.5]
   \draw [step=1.0,black, very thick] (1,1) grid (7,7);

   \foreach \x/\y in {1/1,1/3,1/5,2/2,2/4,2/6,3/1,3/3,3/5,4/2,4/4,4/6,5/1,5/3,5/5,6/2,6/4,6/6}
      \fill[draw=black] (\x,\y) rectangle ++(1,1);
   \node[below,font=\small] at (s6618.south) {$m=18$};
\end{scope}

\end{tikzpicture}
  \caption{Optimal arrangements on a $6\times 6$ grid}\label{6figure}
  \end{center}
\end{figure}

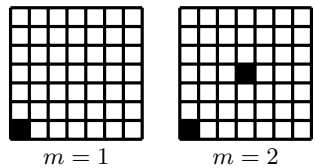
\begin{figure}
  \centering
  \begin{tikzpicture}[scale = 0.5]
\begin{scope}[local bounding box=s771, xshift = -2.25cm, scale = 0.5]

   \draw [step=1.0,black, very thick] (1,1) grid (8,8);

   \foreach \x/\y in {1/1}
      \fill[draw=black] (\x,\y) rectangle ++(1,1);
    \node[below,font=\small] at (s771.south) {$m=1$};
\end{scope}
\begin{scope}[local bounding box=s772, xshift = 2.25cm, scale = 0.5]
\draw [step=1.0,black, very thick] (1,1) grid (8,8);

   \foreach \x/\y in {1/1,4/4}
      \fill[draw=black] (\x,\y) rectangle ++(1,1);
   \node[below,font=\small] at (s772.south) {$m=2$};
\end{scope}

\end{tikzpicture}
  \caption{Optimal arrangements on a $7\times 7$ grid}\label{7figure}
\end{figure}

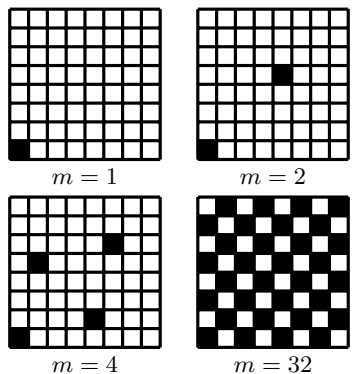
\begin{figure}
  \centering
  \begin{tikzpicture}[scale = 0.5]
\begin{scope}[local bounding box=s881, xshift = -2.5cm, scale = 0.5]

   \draw [step=1.0,black, very thick] (1,1) grid (9,9);

   \foreach \x/\y in {1/1}
      \fill[draw=black] (\x,\y) rectangle ++(1,1);
    \node[below,font=\small] at (s881.south) {$m=1$};
\end{scope}
\begin{scope}[local bounding box=s882, xshift = 2.5cm, scale = 0.5]
\draw [step=1.0,black, very thick] (1,1) grid (9,9);

   \foreach \x/\y in {1/1,5/5}
      \fill[draw=black] (\x,\y) rectangle ++(1,1);
   \node[below,font=\small] at (s882.south) {$m=2$};
\end{scope}
\begin{scope}[local bounding box=s884, xshift = -2.5cm, yshift = -5cm, scale = 0.5]
   \draw [step=1.0,black, very thick] (1,1) grid (9,9);

   \foreach \x/\y in {1/1, 2/5, 5/2, 6/6}
      \fill[draw=black] (\x,\y) rectangle ++(1,1);
   \node[below,font=\small] at (s884.south) {$m=4$};
\end{scope}
\begin{scope}[local bounding box=s8832, xshift = 2.5cm, yshift = -5cm, scale = 0.5]
   \draw [step=1.0,black, very thick] (1,1) grid (9,9);

   \foreach \x/\y in {1/1,1/3,1/5,1/7,2/2,2/4,2/6,2/8,3/1,3/3,3/5,3/7,4/2,4/4,4/6,4/8,5/1,5/3,5/5,5/7,6/2,6/4,6/6,6/8,7/1,7/3,7/5,7/7,8/2,8/4,8/6,8/8}
      \fill[draw=black] (\x,\y) rectangle ++(1,1);
   \node[below,font=\small] at (s8832.south) {$m=32$};
\end{scope}

\end{tikzpicture}
  \caption{Optimal arrangements on a $8\times 8$ grid}\label{8figure}
\end{figure}

\begin{figure}
  \centering
  \begin{tikzpicture}[scale = 0.5]
\begin{scope}[local bounding box=s111, xshift = -3cm, scale = 0.5]

   \draw [step=1.0,black, very thick] (1,1) grid (11,11);

   \foreach \x/\y in {1/1}
      \fill[draw=black] (\x,\y) rectangle ++(1,1);
    \node[below,font=\small] at (s111.south) {$m=1$};
\end{scope}
\begin{scope}[local bounding box=s112, xshift = 3cm, scale = 0.5]
\draw [step=1.0,black, very thick] (1,1) grid (11,11);

   \foreach \x/\y in {1/1,6/6}
      \fill[draw=black] (\x,\y) rectangle ++(1,1);
   \node[below,font=\small] at (s112.south) {$m=2$};
\end{scope}
\begin{scope}[local bounding box=s114, xshift = -3cm, yshift = -6cm, scale = 0.5]
   \draw [step=1.0,black, very thick] (1,1) grid (11,11);

   \foreach \x/\y in {1/1, 3/6, 6/2, 8/7}
      \fill[draw=black] (\x,\y) rectangle ++(1,1);
   \node[below,font=\small] at (s114.south) {$m=4$};
\end{scope}
\begin{scope}[local bounding box=s1120, xshift = 3cm, yshift = -6cm, scale = 0.5]
   \draw [step=1.0,black, very thick] (1,1) grid (11,11);

   \foreach \x/\y in {1/1, 1/6, 2/4, 2/9, 3/7, 3/2 ,4/10, 4/5, 5/3, 5/8, 6/6, 6/1, 7/9, 7/4, 8/2, 8/7, 9/5, 9/10, 10/8, 10/3}
      \fill[draw=black] (\x,\y) rectangle ++(1,1);
   \node[below,font=\small] at (s1120.south) {$m=20$};
\end{scope}
\begin{scope}[local bounding box=s1150, xshift = 0cm, yshift = -12cm, scale = 0.5]
   \draw [step=1.0,black, very thick] (1,1) grid (11,11);

   \foreach \x/\y in {1/1,1/3,1/5,1/7,1/9,2/2,2/4,2/6,2/8,2/10,3/1,3/3,3/5,3/7,3/9, 4/2,4/4,4/6,4/8,4/10,5/1,5/3,5/5,5/7,5/9,6/2,6/4,6/6,6/8,6/10,7/1,7/3,7/5,7/7,7/9,8/2,8/4,8/6,8/8,8/10,9/1,9/3,9/5,9/7,9/9,10/2,10/4,10/6,10/8,10/10}
      \fill[draw=black] (\x,\y) rectangle ++(1,1);
   \node[below,font=\small] at (s1150.south) {$m=50$};
\end{scope}

\end{tikzpicture}
  \caption{Optimal arrangements on a $10\times 10$ grid}\label{10figure}
\end{figure}
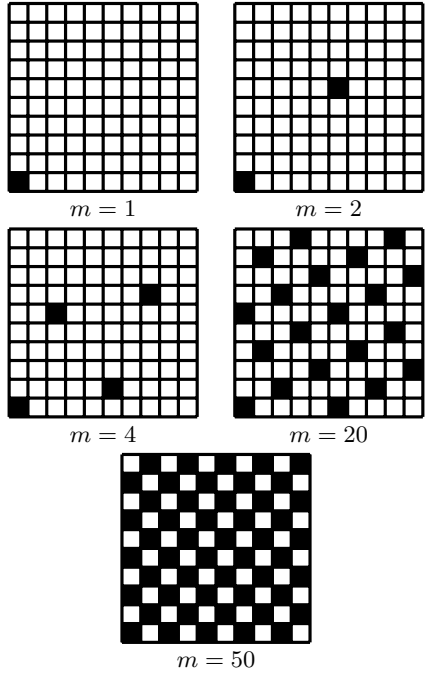

   Note that we could find several new optimal configurations by computing the SDP bound $SDPQAP(A,B)$, for example the case $m=20$ on a $10 \times 10$ grid.
   We do not include results for the weaker SDP bound $SDPPB(A,B)$ in the tables, since these turned out to equal the projected eigenvalue bound
   for small instances. We do not know if these bounds coincide in general, though.

   {\color{black} In these small cases we noticed that the bound is sharp only in (some of the) cases where $m$
   divides $n_1n_2$, and are, except in the case $m=4$ for some choices of $n_1$ and $n_2$, given by lattices. In these cases
    the nonzero variables $y_{ij}^{b\to b}$ do actually hint as to how the lattice can be constructed. For example in the case
    $n_1=n_2=10$, $m=20$ the optimal solution has $y_{ij}^{b\to b} =0$ except for $y_{05}^{b\to b}=y_{55}^{b\to b} = 0.2$
    and $y_{12}^{b\to b}=y_{13}^{b\to b}=y_{24}^{b\to b}=y_{34}^{b\to b}=0.4$, and their symmetric counterparts.
    Compare these to the drawing in Figure \ref{10figure}:
   These are exactly the distances appearing in the drawing, namely $(0,5)$, $(5,5)$, $(1,2)$, $(1,3)$, $(2,4)$, and $(3,4)$.
       We can use this knowledge to construct the solution as follows.
        The size of the orbit of a single point, repeatedly shifting it in direction $(i,j)$,
         is $\mathrm{lcm}\left(\frac{\mathrm{lcm}(n_1,i}{i},\frac{\mathrm{lcm}(n_2,j}{j}\right)$,
          where $\mathrm{lcm}(a,b)$ denotes the least common multiple of $a$ and $b$. We say that two lattice directions are
          \emph{distinct}, if the corresponding orbits of the same point overlap only in that point.
          In this example the orbit sizes are $2$ for the directions $(0,5)$ and $(5,5)$, $5$ for $(2,4)$, and $10$ for $(1,2)$, $(1,3)$ and $(3,4)$.
          If we now choose pairwise distinct orbits of orbit sizes factoring $m$,
          we can reconstruct the lattice solution by taking repeated orbits. In this case we can construct the solution in
          Figure \ref{10figure} for $n_1=n_2=10$ and $m=20$ by choosing the triple of generators
           $\{(0,5),(5,5),(2,4)\}$ or any of the pairs $\{(1,2),(0,5)\}, \{(1,2),(5,5)\}, \{(1,3),(0,5)\}, \{(3,4),(5,5)\}$.
            Alternatively we can find a second solution (which is the same, but mirrored), by choosing swapping the two axes of each generator.

   This way one can rapidly find cases where the bound is sharp. In about 30 minutes we were able to identify $2382$ such cases (of the $15083$ cases where $m$ divides $n_1n_2$), from grid sizes $1\times 1$ to $50\times 50$. We list a few such cases in Figure \ref{bigGridFigure}.
   }

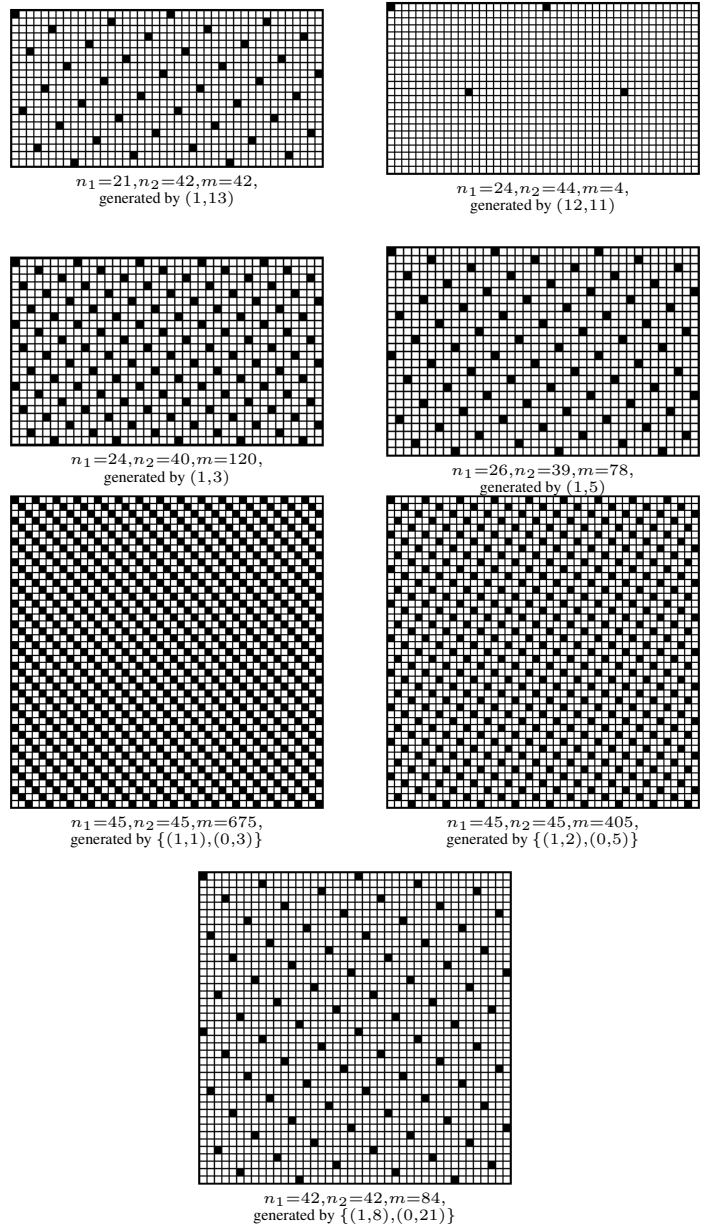
\begin{figure}
  \centering
  \begin{tikzpicture}[scale = 0.5]
\begin{scope}[local bounding box=s111, xshift = -5cm, scale = 0.5]
    \node[inner sep=0pt,draw, very thick] at (0,0)
    {\includegraphics[width=.25\textwidth]{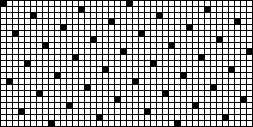}};
   \node[below,font=\small] at (s111.south) {$\substack{n_1=21, n_2 = 42, m = 42,\\ \text{generated by }(1,13)}$};
\end{scope}
\begin{scope}[local bounding box=s112, xshift = 5cm, scale = 0.5]
\node[inner sep=0pt,draw, very thick] at (0,0)
    {\includegraphics[width=.25\textwidth]{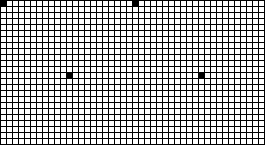}};
   \node[below,font=\small] at (s112.south) {$\substack{n_1=24, n_2 = 44, m = 4,\\ \text{generated by }(12,11)}$};
\end{scope}
\begin{scope}[local bounding box=s114, xshift = -5cm, yshift = -7cm, scale = 0.5]
\node[inner sep=0pt,draw, very thick] at (0,0)
    {\includegraphics[width=.25\textwidth]{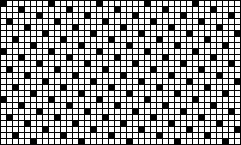}};
   \node[below,font=\small] at (s114.south) {$\substack{n_1=24, n_2 = 40, m = 120, \\ \text{generated by }(1,3)}$};
\end{scope}
\begin{scope}[local bounding box=s1120, xshift = 5cm, yshift = -7cm, scale = 0.5]
\node[inner sep=0pt,draw, very thick] at (0,0)
    {\includegraphics[width=.25\textwidth]{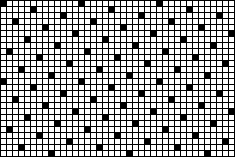}};
   \node[below,font=\small] at (s1120.south) {$\substack{n_1=26, n_2 = 39, m = 78,\\ \text{generated by }(1,5)}$};
\end{scope}
\begin{scope}[local bounding box=s1150, xshift = -5cm, yshift = -15cm, scale = 0.5]
\node[inner sep=0pt,draw, very thick] at (0,0)
    {\includegraphics[width=.25\textwidth]{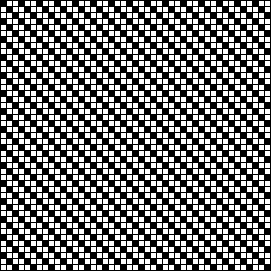}};
   \node[below,font=\small] at (s1150.south) {$\substack{n_1=45, n_2 = 45, m = 675, \\ \text{generated by }\{(1,1),(0,3)\}}$};
\end{scope}
\begin{scope}[local bounding box=s1160, xshift = 5cm, yshift = -15cm, scale = 0.5]
\node[inner sep=0pt,draw, very thick] at (0,0)
    {\includegraphics[width=.25\textwidth]{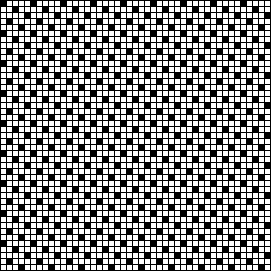}};
   \node[below,font=\small] at (s1160.south) {$\substack{n_1=45, n_2 = 45, m = 405, \\ \text{generated by }\{(1,2),(0,5)\}}$};
\end{scope}
\begin{scope}[local bounding box=s1170, xshift = 0cm, yshift = -25cm, scale = 0.5]
\node[inner sep=0pt,draw, very thick] at (0,0)
    {\includegraphics[width=.25\textwidth]{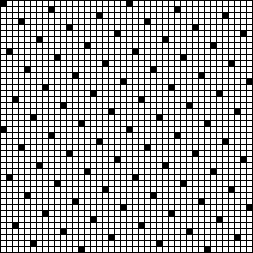}};
   \node[below,font=\small] at (s1170.south) {$\substack{n_1=42, n_2 = 42, m = 84, \\ \text{generated by }\{(1,8),(0,21)\}}$};
\end{scope}

\end{tikzpicture}
  \caption{A small selection of cases where the SDP bound is sharp for larger grids.}\label{bigGridFigure}
\end{figure}

{\color{black}
\section{Concluding remarks}
We have been able to use semidefinite programming bounds to identify several new minimum energy configurations
of repulsing particles on a toric grid, as shown in Figure \ref{bigGridFigure}, for example.
The next step would be to use our insight on optimal lattice configurations that we obtained numerically, to prove optimality for the generalised families
of corresponding lattice configurations, similar to the way the authors of \cite{bouman2013energy} proved optimality of certain `chessboard' configurations.
This remains a topic for future research.}

\section*{Acknowledgement}
The authors thank Henry Wolkowicz for pointing out the reference \cite{pong2016eigenvalue}, and for helpful comments.


\appendix
{
\twocolumn
\section{List of sharp cases}\label{SharpList}
Here we give a list of cases where the bound turned out the be sharp (with $m\geq 2$). The fourth column gives a generator of the optimal solution: Start with a single particle in any position, then repeatedly apply the shifts given in the last column to generate the full solution.

\footnotesize

\begin{xtabular}{|c|c|c|l|}
  \hline
  $n_1$ & $n_2$ & $m$ & generator \\
  \hline
  2 &2& 2 &(1, 1), \\
3 &3& 3 &(1, 1), \\
4 &1& 2 &(2, 0), \\
4 &2& 2 &(2, 1), \\
4 &2& 4 &(1, 1), \\
4 &4& 2 &(2, 2), \\
4 &4& 8 &(1, 1), (0, 2), \\
5 &5& 5 &(1, 2), \\
6 &1& 2 &(3, 0), \\
6 &1& 3 &(2, 0), \\
6 &2& 2 &(3, 1), \\
6 &2& 6 &(1, 1), \\
6 &3& 3 &(2, 1), \\
6 &3& 6 &(1, 1), \\
6 &4& 2 &(3, 2), \\
6 &4& 12 &(1, 1), \\
6 &6& 2 &(3, 3), \\
6 &6& 12 &(1, 1), (0, 3), \\
6 &6& 18 &(1, 1), (0, 2), \\
8 &1& 2 &(4, 0), \\
8 &1& 4 &(2, 0), \\
8 &2& 2 &(4, 1), \\
8 &2& 4 &(2, 1), \\
8 &2& 8 &(1, 1), \\
8 &4& 2 &(4, 2), \\
8 &4& 4 &(2, 2), \\
8 &4& 16 &(1, 1), (0, 2), \\
8 &6& 2 &(4, 3), \\
8 &6& 24 &(1, 1), \\
8 &8& 2 &(4, 4), \\
8 &8& 32 &(1, 1), (0, 2), \\
9 &1& 3 &(3, 0), \\
9 &3& 3 &(3, 1), \\
9 &3& 9 &(1, 1), \\
9 &6& 18 &(1, 1), \\
9 &9& 27 &(1, 1), (0, 3), \\
10 &1& 2 &(5, 0), \\
10 &1& 5 &(2, 0), \\
10 &2& 2 &(5, 1), \\
10 &2& 10 &(1, 1), \\
10 &4& 2 &(5, 2), \\
10 &4& 20 &(1, 1), \\
10 &5& 10 &(1, 2), \\
10 &6& 2 &(5, 3), \\
10 &6& 30 &(1, 1), \\
10 &8& 2 &(5, 4), \\
10 &8& 40 &(1, 1), \\
10 &10& 2 &(5, 5), \\
10 &10& 20 &(1, 2), (0, 5), \\
10 &10& 50 &(1, 1), (0, 2), \\
12 &1& 2 &(6, 0), \\
12 &1& 4 &(3, 0), \\
12 &1& 6 &(2, 0), \\
12 &2& 2 &(6, 1), \\
12 &2& 4 &(3, 1), \\
12 &2& 6 &(2, 1), \\
12 &2& 12 &(1, 1), \\
12 &3& 6 &(2, 1), \\
12 &3& 12 &(1, 1), \\
12 &4& 2 &(6, 2), \\
12 &4& 4 &(3, 2), \\
12 &4& 24 &(1, 1), (0, 2), \\
12 &6& 2 &(6, 3), \\
12 &6& 4 &(3, 3), \\
12 &6& 24 &(1, 1), (0, 3), \\
12 &6& 36 &(1, 1), (0, 2), \\
12 &8& 2 &(6, 4), \\
12 &8& 4 &(3, 4), \\
12 &8& 48 &(1, 1), (0, 4), (6, 0), \\
12 &9& 36 &(1, 1), \\
12 &10& 2 &(6, 5), \\
12 &10& 60 &(1, 1), \\
12 &12& 2 &(6, 6), \\
12 &12& 48 &(1, 1), (0, 3), \\
12 &12& 72 &(1, 1), (0, 2), \\
13 &13& 13 &(1, 5), \\
14 &1& 2 &(7, 0), \\
14 &1& 7 &(2, 0), \\
14 &2& 2 &(7, 1), \\
14 &2& 14 &(1, 1), \\
14 &4& 2 &(7, 2), \\
14 &4& 28 &(1, 1), \\
14 &6& 2 &(7, 3), \\
14 &6& 42 &(1, 1), \\
14 &8& 2 &(7, 4), \\
14 &8& 56 &(1, 1), \\
14 &10& 2 &(7, 5), \\
14 &10& 70 &(1, 1), \\
14 &12& 2 &(7, 6), \\
14 &12& 84 &(1, 1), \\
14 &14& 2 &(7, 7), \\
14 &14& 98 &(1, 1), (0, 2), \\
15 &1& 3 &(5, 0), \\
15 &1& 5 &(3, 0), \\
15 &3& 15 &(1, 1), \\
15 &5& 5 &(3, 2), \\
15 &5& 15 &(1, 2), \\
15 &6& 30 &(1, 1), \\
15 &9& 45 &(1, 1), \\
15 &10& 30 &(1, 3), \\
15 &12& 60 &(1, 1), \\
15 &15& 45 &(1, 2), (0, 5), \\
15 &15& 75 &(1, 1), (0, 3), \\
16 &1& 2 &(8, 0), \\
16 &1& 4 &(4, 0), \\
16 &1& 8 &(2, 0), \\
16 &2& 2 &(8, 1), \\
16 &2& 4 &(4, 1), \\
16 &2& 8 &(2, 1), \\
16 &2& 16 &(1, 1), \\
16 &4& 2 &(8, 2), \\
16 &4& 4 &(4, 2), \\
16 &4& 32 &(1, 1), (0, 2), \\
16 &6& 2 &(8, 3), \\
16 &6& 4 &(4, 3), \\
16 &6& 48 &(1, 1), \\
16 &8& 2 &(8, 4), \\
16 &8& 4 &(4, 4), \\
16 &8& 16 &(1, 3), \\
16 &8& 64 &(1, 1), (0, 2), \\
16 &10& 2 &(8, 5), \\
16 &10& 80 &(1, 1), \\
16 &12& 2 &(8, 6), \\
16 &12& 96 &(1, 1), (0, 6), \\
16 &14& 2 &(8, 7), \\
16 &14& 112 &(1, 1), \\
16 &16& 2 &(8, 8), \\
16 &16& 32 &(1, 3), (0, 8), \\
16 &16& 128 &(1, 1), (0, 2), \\
18 &1& 2 &(9, 0), \\
18 &1& 6 &(3, 0), \\
18 &1& 9 &(2, 0), \\
18 &2& 2 &(9, 1), \\
18 &2& 6 &(3, 1), \\
18 &2& 18 &(1, 1), \\
18 &3& 6 &(3, 1), \\
18 &3& 18 &(1, 1), \\
18 &4& 2 &(9, 2), \\
18 &4& 36 &(1, 1), \\
18 &6& 2 &(9, 3), \\
18 &6& 36 &(1, 1), (0, 3), \\
18 &6& 54 &(1, 1), (0, 2), \\
18 &8& 2 &(9, 4), \\
18 &8& 72 &(1, 1), \\
18 &9& 54 &(1, 1), (0, 3), \\
18 &10& 2 &(9, 5), \\
18 &10& 90 &(1, 1), \\
18 &12& 2 &(9, 6), \\
18 &12& 72 &(1, 1), (0, 6), (9, 0), \\
18 &12& 108 &(1, 1), (0, 4), \\
18 &14& 2 &(9, 7), \\
18 &14& 126 &(1, 1), \\
18 &15& 90 &(1, 1), \\
18 &16& 2 &(9, 8), \\
18 &16& 144 &(1, 1), \\
18 &18& 2 &(9, 9), \\
18 &18& 108 &(1, 1), (0, 3), \\
18 &18& 162 &(1, 1), (0, 2), \\
20 &1& 2 &(10, 0), \\
20 &1& 4 &(5, 0), \\
20 &1& 5 &(4, 0), \\
20 &1& 10 &(2, 0), \\
20 &2& 2 &(10, 1), \\
20 &2& 4 &(5, 1), \\
20 &2& 10 &(2, 1), \\
20 &2& 20 &(1, 1), \\
20 &4& 2 &(10, 2), \\
20 &4& 4 &(5, 2), \\
20 &4& 40 &(1, 1), (0, 2), \\
20 &5& 10 &(2, 2), \\
20 &5& 20 &(1, 2), \\
20 &6& 2 &(10, 3), \\
20 &6& 4 &(5, 3), \\
20 &6& 60 &(1, 1), \\
20 &8& 2 &(10, 4), \\
20 &8& 4 &(5, 4), \\
20 &8& 80 &(1, 1), (0, 4), (10, 0), \\
20 &10& 2 &(10, 5), \\
20 &10& 4 &(5, 5), \\
20 &10& 40 &(1, 2), (0, 5), \\
20 &10& 100 &(1, 1), (0, 2), \\
20 &12& 2 &(10, 6), \\
20 &12& 4 &(5, 6), \\
20 &12& 120 &(1, 1), (0, 6), \\
20 &14& 140 &(1, 1), \\
20 &15& 60 &(1, 2), \\
20 &16& 2 &(10, 8), \\
20 &16& 160 &(1, 1), (0, 8), (10, 0), \\
20 &18& 2 &(10, 9), \\
20 &18& 180 &(1, 1), \\
20 &20& 2 &(10, 10), \\
20 &20& 80 &(1, 2), (0, 5), \\
20 &20& 200 &(1, 1), (0, 2), \\
21 &1& 3 &(7, 0), \\
21 &1& 7 &(3, 0), \\
21 &3& 21 &(1, 1), \\
21 &6& 42 &(1, 1), \\
21 &9& 63 &(1, 1), \\
21 &12& 84 &(1, 1), \\
21 &15& 105 &(1, 1), \\
21 &18& 126 &(1, 1), \\
21 &21& 21 &(1, 8), \\
21 &21& 147 &(1, 1), (0, 3), \\
22 &1& 2 &(11, 0), \\
22 &1& 11 &(2, 0), \\
22 &2& 2 &(11, 1), \\
22 &2& 22 &(1, 1), \\
22 &4& 2 &(11, 2), \\
22 &4& 44 &(1, 1), \\
22 &6& 2 &(11, 3), \\
22 &6& 66 &(1, 1), \\
22 &8& 2 &(11, 4), \\
22 &8& 88 &(1, 1), \\
22 &10& 2 &(11, 5), \\
22 &10& 110 &(1, 1), \\
22 &12& 2 &(11, 6), \\
22 &12& 132 &(1, 1), \\
22 &14& 2 &(11, 7), \\
22 &14& 154 &(1, 1), \\
22 &16& 2 &(11, 8), \\
22 &16& 176 &(1, 1), \\
22 &18& 2 &(11, 9), \\
22 &18& 198 &(1, 1), \\
22 &20& 2 &(11, 10), \\
22 &20& 220 &(1, 1), \\
22 &22& 2 &(11, 11), \\
22 &22& 242 &(1, 1), (0, 2), \\
24 &1& 2 &(12, 0), \\
24 &1& 4 &(6, 0), \\
24 &1& 8 &(3, 0), \\
24 &1& 6 &(4, 0), \\
24 &1& 12 &(2, 0), \\
24 &2& 2 &(12, 1), \\
24 &2& 4 &(6, 1), \\
24 &2& 8 &(3, 1), \\
24 &2& 6 &(4, 1), \\
24 &2& 12 &(2, 1), \\
24 &2& 24 &(1, 1), \\
24 &3& 6 &(4, 1), \\
24 &3& 12 &(2, 1), \\
24 &3& 24 &(1, 1), \\
24 &4& 2 &(12, 2), \\
24 &4& 4 &(6, 2), \\
24 &4& 6 &(4, 2), \\
24 &4& 48 &(1, 1), (0, 2), \\
24 &6& 2 &(12, 3), \\
24 &6& 4 &(6, 3), \\
24 &6& 48 &(1, 1), (0, 3), \\
24 &6& 72 &(1, 1), (0, 2), \\
24 &8& 2 &(12, 4), \\
24 &8& 4 &(6, 4), \\
24 &8& 96 &(1, 1), (0, 2), \\
24 &9& 72 &(1, 1), \\
24 &10& 2 &(12, 5), \\
24 &10& 4 &(6, 5), \\
24 &10& 120 &(1, 1), \\
24 &12& 2 &(12, 6), \\
24 &12& 4 &(6, 6), \\
24 &12& 96 &(1, 1), (0, 3), \\
24 &12& 144 &(1, 1), (0, 2), \\
24 &14& 2 &(12, 7), \\
24 &14& 168 &(1, 1), \\
24 &15& 120 &(1, 1), \\
24 &16& 2 &(12, 8), \\
24 &16& 48 &(1, 3), \\
24 &16& 192 &(1, 1), (0, 4), (6, 0), \\
24 &18& 2 &(12, 9), \\
24 &18& 144 &(1, 1), (0, 9), \\
24 &18& 216 &(1, 1), (0, 6), (8, 0), \\
24 &20& 2 &(12, 10), \\
24 &20& 240 &(1, 1), (0, 10), \\
24 &21& 168 &(1, 1), \\
24 &22& 2 &(12, 11), \\
24 &22& 264 &(1, 1), \\
24 &24& 2 &(12, 12), \\
24 &24& 192 &(1, 1), (0, 3), \\
24 &24& 72 &(1, 3), (0, 8), \\
24 &24& 288 &(1, 1), (0, 2), \\
25 &1& 5 &(5, 0), \\
25 &5& 25 &(1, 2), \\
25 &10& 50 &(1, 3), \\
25 &15& 75 &(1, 2), \\
25 &20& 100 &(1, 3), \\
25 &25& 125 &(1, 2), (0, 5), \\
26 &1& 2 &(13, 0), \\
26 &1& 13 &(2, 0), \\
26 &2& 2 &(13, 1), \\
26 &2& 26 &(1, 1), \\
26 &4& 2 &(13, 2), \\
26 &4& 52 &(1, 1), \\
26 &6& 2 &(13, 3), \\
26 &6& 78 &(1, 1), \\
26 &8& 2 &(13, 4), \\
26 &8& 104 &(1, 1), \\
26 &10& 2 &(13, 5), \\
26 &10& 130 &(1, 1), \\
26 &12& 2 &(13, 6), \\
26 &12& 156 &(1, 1), \\
26 &13& 26 &(1, 5), \\
26 &14& 2 &(13, 7), \\
26 &14& 182 &(1, 1), \\
26 &16& 2 &(13, 8), \\
26 &16& 208 &(1, 1), \\
26 &18& 2 &(13, 9), \\
26 &18& 234 &(1, 1), \\
26 &20& 2 &(13, 10), \\
26 &20& 260 &(1, 1), \\
26 &22& 2 &(13, 11), \\
26 &22& 286 &(1, 1), \\
26 &24& 2 &(13, 12), \\
26 &24& 312 &(1, 1), \\
26 &26& 2 &(13, 13), \\
26 &26& 52 &(1, 5), (0, 13), \\
26 &26& 338 &(1, 1), (0, 2), \\
27 &1& 3 &(9, 0), \\
27 &1& 9 &(3, 0), \\
27 &3& 9 &(3, 1), \\
27 &3& 27 &(1, 1), \\
27 &6& 54 &(1, 1), \\
27 &9& 81 &(1, 1), (0, 3), \\
27 &12& 108 &(1, 1), \\
27 &15& 135 &(1, 1), \\
27 &18& 162 &(1, 1), (0, 6), \\
27 &21& 189 &(1, 1), \\
27 &24& 216 &(1, 1), \\
27 &27& 243 &(1, 1), (0, 3), \\
28 &1& 2 &(14, 0), \\
28 &1& 4 &(7, 0), \\
28 &1& 7 &(4, 0), \\
28 &1& 14 &(2, 0), \\
28 &2& 2 &(14, 1), \\
28 &2& 4 &(7, 1), \\
28 &2& 14 &(2, 1), \\
28 &2& 28 &(1, 1), \\
28 &4& 2 &(14, 2), \\
28 &4& 4 &(7, 2), \\
28 &4& 56 &(1, 1), (0, 2), \\
28 &6& 2 &(14, 3), \\
28 &6& 4 &(7, 3), \\
28 &6& 84 &(1, 1), \\
28 &7& 14 &(2, 3), \\
28 &8& 2 &(14, 4), \\
28 &8& 4 &(7, 4), \\
28 &8& 112 &(1, 1), (0, 4), (14, 0), \\
28 &10& 2 &(14, 5), \\
28 &10& 4 &(7, 5), \\
28 &10& 140 &(1, 1), \\
28 &12& 2 &(14, 6), \\
28 &12& 4 &(7, 6), \\
28 &12& 168 &(1, 1), (0, 6), \\
28 &14& 2 &(14, 7), \\
28 &14& 4 &(7, 7), \\
28 &14& 196 &(1, 1), (0, 2), \\
28 &16& 2 &(14, 8), \\
28 &16& 4 &(7, 8), \\
28 &16& 224 &(1, 1), (0, 8), (14, 0), \\
28 &18& 2 &(14, 9), \\
28 &18& 252 &(1, 1), \\
28 &20& 2 &(14, 10), \\
28 &20& 280 &(1, 1), (0, 10), \\
28 &22& 2 &(14, 11), \\
28 &22& 308 &(1, 1), \\
28 &24& 2 &(14, 12), \\
28 &24& 336 &(1, 1), (0, 12), (14, 0), \\
28 &26& 2 &(14, 13), \\
28 &26& 364 &(1, 1), \\
28 &28& 2 &(14, 14), \\
28 &28& 392 &(1, 1), (0, 2), \\
30 &1& 2 &(15, 0), \\
30 &1& 6 &(5, 0), \\
30 &1& 10 &(3, 0), \\
30 &1& 15 &(2, 0), \\
30 &2& 2 &(15, 1), \\
30 &2& 6 &(5, 1), \\
30 &2& 10 &(3, 1), \\
30 &2& 30 &(1, 1), \\
30 &3& 6 &(5, 1), \\
30 &3& 30 &(1, 1), \\
30 &4& 2 &(15, 2), \\
30 &4& 6 &(5, 2), \\
30 &4& 60 &(1, 1), \\
30 &5& 10 &(3, 2), \\
30 &5& 30 &(1, 2), \\
30 &6& 2 &(15, 3), \\
30 &6& 60 &(1, 1), (0, 3), \\
30 &6& 90 &(1, 1), (0, 2), \\
30 &8& 2 &(15, 4), \\
30 &8& 120 &(1, 1), \\
30 &9& 90 &(1, 1), \\
30 &10& 2 &(15, 5), \\
30 &10& 60 &(1, 2), (0, 5), \\
30 &10& 150 &(1, 1), (0, 2), \\
30 &12& 2 &(15, 6), \\
30 &12& 120 &(1, 1), (0, 6), (15, 0), \\
30 &12& 180 &(1, 1), (0, 4), \\
30 &14& 2 &(15, 7), \\
30 &14& 210 &(1, 1), \\
30 &15& 90 &(1, 2), (0, 5), \\
30 &15& 150 &(1, 1), (0, 3), \\
30 &16& 2 &(15, 8), \\
30 &16& 240 &(1, 1), \\
30 &18& 2 &(15, 9), \\
30 &18& 180 &(1, 1), (0, 9), \\
30 &18& 270 &(1, 1), (0, 6), (10, 0), \\
30 &20& 2 &(15, 10), \\
30 &20& 120 &(1, 3), (0, 10), (15, 0), \\
30 &20& 300 &(1, 1), (0, 4), \\
30 &21& 210 &(1, 1), \\
30 &22& 2 &(15, 11), \\
30 &22& 330 &(1, 1), \\
30 &24& 2 &(15, 12), \\
30 &24& 240 &(1, 1), (0, 12), (15, 0), \\
30 &24& 360 &(1, 1), (0, 8), \\
30 &25& 150 &(1, 2), \\
30 &26& 2 &(15, 13), \\
30 &26& 390 &(1, 1), \\
30 &27& 270 &(1, 1), \\
30 &28& 2 &(15, 14), \\
30 &28& 420 &(1, 1), \\
30 &30& 2 &(15, 15), \\
30 &30& 180 &(1, 2), (0, 5), \\
30 &30& 300 &(1, 1), (0, 3), \\
30 &30& 450 &(1, 1), (0, 2), \\
32 &1& 2 &(16, 0), \\
32 &1& 4 &(8, 0), \\
32 &1& 8 &(4, 0), \\
32 &1& 16 &(2, 0), \\
32 &2& 2 &(16, 1), \\
32 &2& 4 &(8, 1), \\
32 &2& 8 &(4, 1), \\
32 &2& 16 &(2, 1), \\
32 &2& 32 &(1, 1), \\
32 &4& 2 &(16, 2), \\
32 &4& 4 &(8, 2), \\
32 &4& 8 &(4, 2), \\
32 &4& 64 &(1, 1), (0, 2), \\
32 &6& 2 &(16, 3), \\
32 &6& 4 &(8, 3), \\
32 &6& 96 &(1, 1), \\
32 &8& 2 &(16, 4), \\
32 &8& 4 &(8, 4), \\
32 &8& 32 &(1, 3), \\
32 &8& 128 &(1, 1), (0, 2), \\
32 &10& 2 &(16, 5), \\
32 &10& 4 &(8, 5), \\
32 &10& 160 &(1, 1), \\
32 &12& 2 &(16, 6), \\
32 &12& 4 &(8, 6), \\
32 &12& 192 &(1, 1), (0, 6), \\
32 &14& 2 &(16, 7), \\
32 &14& 4 &(8, 7), \\
32 &14& 224 &(1, 1), \\
32 &16& 2 &(16, 8), \\
32 &16& 4 &(8, 8), \\
32 &16& 64 &(1, 3), (0, 8), \\
32 &16& 256 &(1, 1), (0, 2), \\
32 &18& 2 &(16, 9), \\
32 &18& 288 &(1, 1), \\
32 &20& 2 &(16, 10), \\
32 &20& 320 &(1, 1), (0, 10), \\
32 &22& 2 &(16, 11), \\
32 &22& 352 &(1, 1), \\
32 &24& 2 &(16, 12), \\
32 &24& 96 &(1, 5), \\
32 &24& 384 &(1, 1), (0, 6), \\
32 &26& 2 &(16, 13), \\
32 &26& 416 &(1, 1), \\
32 &28& 2 &(16, 14), \\
32 &28& 448 &(1, 1), (0, 14), \\
32 &30& 2 &(16, 15), \\
32 &30& 480 &(1, 1), \\
32 &32& 2 &(16, 16), \\
32 &32& 128 &(1, 3), (0, 8), \\
32 &32& 512 &(1, 1), (0, 2), \\
33 &1& 11 &(3, 0), \\
33 &3& 33 &(1, 1), \\
33 &6& 66 &(1, 1), \\
33 &9& 99 &(1, 1), \\
33 &12& 132 &(1, 1), \\
33 &15& 165 &(1, 1), \\
33 &18& 198 &(1, 1), \\
33 &21& 231 &(1, 1), \\
33 &24& 264 &(1, 1), \\
33 &27& 297 &(1, 1), \\
33 &30& 330 &(1, 1), \\
33 &33& 363 &(1, 1), (0, 3), \\
34 &1& 2 &(17, 0), \\
34 &1& 17 &(2, 0), \\
34 &2& 2 &(17, 1), \\
34 &2& 34 &(1, 1), \\
34 &4& 2 &(17, 2), \\
34 &4& 68 &(1, 1), \\
34 &6& 2 &(17, 3), \\
34 &6& 102 &(1, 1), \\
34 &8& 2 &(17, 4), \\
34 &8& 136 &(1, 1), \\
34 &10& 2 &(17, 5), \\
34 &10& 170 &(1, 1), \\
34 &12& 2 &(17, 6), \\
34 &12& 204 &(1, 1), \\
34 &14& 2 &(17, 7), \\
34 &14& 238 &(1, 1), \\
34 &16& 2 &(17, 8), \\
34 &16& 272 &(1, 1), \\
34 &18& 2 &(17, 9), \\
34 &18& 306 &(1, 1), \\
34 &20& 2 &(17, 10), \\
34 &20& 340 &(1, 1), \\
34 &22& 2 &(17, 11), \\
34 &22& 374 &(1, 1), \\
34 &24& 2 &(17, 12), \\
34 &24& 408 &(1, 1), \\
34 &26& 2 &(17, 13), \\
34 &26& 442 &(1, 1), \\
34 &28& 2 &(17, 14), \\
34 &28& 476 &(1, 1), \\
34 &30& 2 &(17, 15), \\
34 &30& 510 &(1, 1), \\
34 &32& 2 &(17, 16), \\
34 &32& 544 &(1, 1), \\
34 &34& 2 &(17, 17), \\
34 &34& 578 &(1, 1), (0, 2), \\
35 &1& 7 &(5, 0), \\
35 &5& 35 &(1, 2), \\
35 &10& 70 &(1, 3), \\
35 &15& 105 &(1, 2), \\
35 &20& 140 &(1, 3), \\
35 &25& 175 &(1, 2), \\
35 &30& 210 &(1, 7), \\
35 &35& 245 &(1, 2), (0, 5), \\
36 &1& 2 &(18, 0), \\
36 &1& 4 &(9, 0), \\
36 &1& 6 &(6, 0), \\
36 &1& 12 &(3, 0), \\
36 &1& 9 &(4, 0), \\
36 &1& 18 &(2, 0), \\
36 &2& 2 &(18, 1), \\
36 &2& 4 &(9, 1), \\
36 &2& 6 &(6, 1), \\
36 &2& 12 &(3, 1), \\
36 &2& 18 &(2, 1), \\
36 &2& 36 &(1, 1), \\
36 &3& 6 &(6, 1), \\
36 &3& 12 &(3, 1), \\
36 &3& 18 &(2, 1), \\
36 &3& 36 &(1, 1), \\
36 &4& 2 &(18, 2), \\
36 &4& 4 &(9, 2), \\
36 &4& 6 &(6, 2), \\
36 &4& 72 &(1, 1), (0, 2), \\
36 &6& 2 &(18, 3), \\
36 &6& 4 &(9, 3), \\
36 &6& 72 &(1, 1), (0, 3), \\
36 &6& 108 &(1, 1), (0, 2), \\
36 &8& 2 &(18, 4), \\
36 &8& 4 &(9, 4), \\
36 &8& 144 &(1, 1), (0, 4), (18, 0), \\
36 &9& 108 &(1, 1), (0, 3), \\
36 &10& 2 &(18, 5), \\
36 &10& 4 &(9, 5), \\
36 &10& 180 &(1, 1), \\
36 &12& 2 &(18, 6), \\
36 &12& 4 &(9, 6), \\
36 &12& 144 &(1, 1), (0, 3), \\
36 &12& 216 &(1, 1), (0, 2), \\
36 &14& 2 &(18, 7), \\
36 &14& 4 &(9, 7), \\
36 &14& 252 &(1, 1), \\
36 &15& 180 &(1, 1), \\
36 &16& 2 &(18, 8), \\
36 &16& 4 &(9, 8), \\
36 &16& 288 &(1, 1), (0, 8), (18, 0), \\
36 &18& 2 &(18, 9), \\
36 &18& 4 &(9, 9), \\
36 &18& 216 &(1, 1), (0, 3), \\
36 &18& 324 &(1, 1), (0, 2), \\
36 &20& 2 &(18, 10), \\
36 &20& 4 &(9, 10), \\
36 &20& 360 &(1, 1), (0, 10), \\
36 &21& 252 &(1, 1), \\
36 &22& 2 &(18, 11), \\
36 &22& 396 &(1, 1), \\
36 &24& 2 &(18, 12), \\
36 &24& 288 &(1, 1), (0, 6), (9, 0), \\
36 &24& 432 &(1, 1), (0, 4), (6, 0), (6, 4), \\
36 &26& 2 &(18, 13), \\
36 &26& 468 &(1, 1), \\
36 &27& 324 &(1, 1), (0, 9), (12, 0), \\
36 &28& 2 &(18, 14), \\
36 &28& 504 &(1, 1), (0, 14), \\
36 &30& 2 &(18, 15), \\
36 &30& 360 &(1, 1), (0, 15), \\
36 &30& 540 &(1, 1), (0, 10), \\
36 &32& 2 &(18, 16), \\
36 &32& 576 &(1, 1), (0, 16), (18, 0), \\
36 &33& 396 &(1, 1), \\
36 &34& 2 &(18, 17), \\
36 &34& 612 &(1, 1), \\
36 &36& 2 &(18, 18), \\
36 &36& 432 &(1, 1), (0, 3), \\
36 &36& 648 &(1, 1), (0, 2), \\
38 &1& 2 &(19, 0), \\
38 &1& 19 &(2, 0), \\
38 &2& 2 &(19, 1), \\
38 &2& 38 &(1, 1), \\
38 &4& 2 &(19, 2), \\
38 &4& 76 &(1, 1), \\
38 &6& 2 &(19, 3), \\
38 &6& 114 &(1, 1), \\
38 &8& 2 &(19, 4), \\
38 &8& 152 &(1, 1), \\
38 &10& 2 &(19, 5), \\
38 &10& 190 &(1, 1), \\
38 &12& 2 &(19, 6), \\
38 &12& 228 &(1, 1), \\
38 &14& 2 &(19, 7), \\
38 &14& 266 &(1, 1), \\
38 &16& 2 &(19, 8), \\
38 &16& 304 &(1, 1), \\
38 &18& 2 &(19, 9), \\
38 &18& 342 &(1, 1), \\
38 &20& 2 &(19, 10), \\
38 &20& 380 &(1, 1), \\
38 &22& 2 &(19, 11), \\
38 &22& 418 &(1, 1), \\
38 &24& 2 &(19, 12), \\
38 &24& 456 &(1, 1), \\
38 &26& 2 &(19, 13), \\
38 &26& 494 &(1, 1), \\
38 &28& 2 &(19, 14), \\
38 &28& 532 &(1, 1), \\
38 &30& 2 &(19, 15), \\
38 &30& 570 &(1, 1), \\
38 &32& 2 &(19, 16), \\
38 &32& 608 &(1, 1), \\
38 &34& 2 &(19, 17), \\
38 &34& 646 &(1, 1), \\
38 &36& 2 &(19, 18), \\
38 &36& 684 &(1, 1), \\
38 &38& 2 &(19, 19), \\
38 &38& 722 &(1, 1), (0, 2), \\
39 &1& 13 &(3, 0), \\
39 &3& 39 &(1, 1), \\
39 &6& 78 &(1, 1), \\
39 &9& 117 &(1, 1), \\
39 &12& 156 &(1, 1), \\
39 &13& 39 &(1, 5), \\
39 &15& 195 &(1, 1), \\
39 &18& 234 &(1, 1), \\
39 &21& 273 &(1, 1), \\
39 &24& 312 &(1, 1), \\
39 &26& 78 &(1, 5), \\
39 &27& 351 &(1, 1), \\
39 &30& 390 &(1, 1), \\
39 &33& 429 &(1, 1), \\
39 &36& 468 &(1, 1), \\
39 &39& 117 &(1, 5), (0, 13), \\
39 &39& 507 &(1, 1), (0, 3), \\
40 &1& 2 &(20, 0), \\
40 &1& 4 &(10, 0), \\
40 &1& 8 &(5, 0), \\
40 &1& 10 &(4, 0), \\
40 &1& 20 &(2, 0), \\
40 &2& 2 &(20, 1), \\
40 &2& 4 &(10, 1), \\
40 &2& 8 &(5, 1), \\
40 &2& 10 &(4, 1), \\
40 &2& 20 &(2, 1), \\
40 &2& 40 &(1, 1), \\
40 &4& 2 &(20, 2), \\
40 &4& 4 &(10, 2), \\
40 &4& 8 &(5, 2), \\
40 &4& 80 &(1, 1), (0, 2), \\
40 &5& 10 &(4, 2), \\
40 &5& 20 &(2, 2), \\
40 &5& 40 &(1, 2), \\
40 &6& 2 &(20, 3), \\
40 &6& 4 &(10, 3), \\
40 &6& 120 &(1, 1), \\
40 &8& 2 &(20, 4), \\
40 &8& 4 &(10, 4), \\
40 &8& 160 &(1, 1), (0, 2), \\
40 &10& 2 &(20, 5), \\
40 &10& 4 &(10, 5), \\
40 &10& 80 &(1, 2), (0, 5), \\
40 &10& 200 &(1, 1), (0, 2), \\
40 &12& 2 &(20, 6), \\
40 &12& 4 &(10, 6), \\
40 &12& 240 &(1, 1), (0, 6), \\
40 &14& 2 &(20, 7), \\
40 &14& 4 &(10, 7), \\
40 &14& 280 &(1, 1), \\
40 &15& 120 &(1, 2), \\
40 &16& 2 &(20, 8), \\
40 &16& 4 &(10, 8), \\
40 &16& 80 &(1, 3), \\
40 &16& 320 &(1, 1), (0, 4), (10, 0), \\
40 &18& 2 &(20, 9), \\
40 &18& 4 &(10, 9), \\
40 &18& 360 &(1, 1), \\
40 &20& 2 &(20, 10), \\
40 &20& 4 &(10, 10), \\
40 &20& 160 &(1, 2), (0, 5), \\
40 &20& 400 &(1, 1), (0, 2), \\
40 &22& 2 &(20, 11), \\
40 &22& 440 &(1, 1), \\
40 &24& 2 &(20, 12), \\
40 &24& 120 &(1, 5), \\
40 &24& 480 &(1, 1), (0, 6), \\
40 &25& 200 &(1, 2), \\
40 &26& 2 &(20, 13), \\
40 &26& 520 &(1, 1), \\
40 &28& 2 &(20, 14), \\
40 &28& 560 &(1, 1), (0, 14), \\
40 &30& 2 &(20, 15), \\
40 &30& 240 &(1, 2), (0, 15), \\
40 &30& 600 &(1, 1), (0, 6), \\
40 &32& 2 &(20, 16), \\
40 &32& 160 &(1, 3), \\
40 &32& 640 &(1, 1), (0, 8), (10, 0), \\
40 &34& 2 &(20, 17), \\
40 &34& 680 &(1, 1), \\
40 &35& 280 &(1, 2), \\
40 &36& 2 &(20, 18), \\
40 &36& 720 &(1, 1), (0, 18), \\
40 &38& 2 &(20, 19), \\
40 &38& 760 &(1, 1), \\
40 &40& 2 &(20, 20), \\
40 &40& 320 &(1, 2), (0, 5), \\
40 &40& 200 &(1, 3), (0, 8), \\
40 &40& 800 &(1, 1), (0, 2), \\
42 &1& 2 &(21, 0), \\
42 &1& 6 &(7, 0), \\
42 &1& 14 &(3, 0), \\
42 &1& 21 &(2, 0), \\
42 &2& 2 &(21, 1), \\
42 &2& 6 &(7, 1), \\
42 &2& 14 &(3, 1), \\
42 &2& 42 &(1, 1), \\
42 &3& 6 &(7, 1), \\
42 &3& 42 &(1, 1), \\
42 &4& 2 &(21, 2), \\
42 &4& 6 &(7, 2), \\
42 &4& 84 &(1, 1), \\
42 &6& 2 &(21, 3), \\
42 &6& 6 &(7, 3), \\
42 &6& 84 &(1, 1), (0, 3), \\
42 &6& 126 &(1, 1), (0, 2), \\
42 &7& 14 &(3, 3), \\
42 &8& 2 &(21, 4), \\
42 &8& 168 &(1, 1), \\
42 &9& 126 &(1, 1), \\
42 &10& 2 &(21, 5), \\
42 &10& 210 &(1, 1), \\
42 &12& 2 &(21, 6), \\
42 &12& 168 &(1, 1), (0, 6), (21, 0), \\
42 &12& 252 &(1, 1), (0, 4), \\
42 &14& 2 &(21, 7), \\
42 &14& 294 &(1, 1), (0, 2), \\
42 &15& 210 &(1, 1), \\
42 &16& 2 &(21, 8), \\
42 &16& 336 &(1, 1), \\
42 &18& 2 &(21, 9), \\
42 &18& 252 &(1, 1), (0, 9), \\
42 &18& 378 &(1, 1), (0, 6), (14, 0), \\
42 &20& 2 &(21, 10), \\
42 &20& 420 &(1, 1), \\
42 &21& 42 &(1, 8), \\
42 &21& 294 &(1, 1), (0, 3), \\
42 &22& 2 &(21, 11), \\
42 &22& 462 &(1, 1), \\
42 &24& 2 &(21, 12), \\
42 &24& 336 &(1, 1), (0, 12), (21, 0), \\
42 &24& 504 &(1, 1), (0, 8), \\
42 &26& 2 &(21, 13), \\
42 &26& 546 &(1, 1), \\
42 &27& 378 &(1, 1), \\
42 &28& 2 &(21, 14), \\
42 &28& 588 &(1, 1), (0, 4), \\
42 &30& 2 &(21, 15), \\
42 &30& 420 &(1, 1), (0, 15), \\
42 &30& 630 &(1, 1), (0, 10), \\
42 &32& 2 &(21, 16), \\
42 &32& 672 &(1, 1), \\
42 &33& 462 &(1, 1), \\
42 &34& 2 &(21, 17), \\
42 &34& 714 &(1, 1), \\
42 &36& 2 &(21, 18), \\
42 &36& 504 &(1, 1), (0, 18), (21, 0), \\
42 &36& 756 &(1, 1), (0, 12), (14, 0), \\
42 &38& 2 &(21, 19), \\
42 &38& 798 &(1, 1), \\
42 &39& 546 &(1, 1), \\
42 &40& 2 &(21, 20), \\
42 &40& 840 &(1, 1), \\
42 &42& 2 &(21, 21), \\
42 &42& 84 &(1, 8), (0, 21), \\
42 &42& 588 &(1, 1), (0, 3), \\
42 &42& 882 &(1, 1), (0, 2), \\
44 &1& 2 &(22, 0), \\
44 &1& 4 &(11, 0), \\
44 &1& 11 &(4, 0), \\
44 &1& 22 &(2, 0), \\
44 &2& 2 &(22, 1), \\
44 &2& 4 &(11, 1), \\
44 &2& 22 &(2, 1), \\
44 &2& 44 &(1, 1), \\
44 &4& 2 &(22, 2), \\
44 &4& 4 &(11, 2), \\
44 &4& 88 &(1, 1), (0, 2), \\
44 &6& 2 &(22, 3), \\
44 &6& 4 &(11, 3), \\
44 &6& 132 &(1, 1), \\
44 &8& 2 &(22, 4), \\
44 &8& 4 &(11, 4), \\
44 &8& 176 &(1, 1), (0, 4), (22, 0), \\
44 &10& 2 &(22, 5), \\
44 &10& 4 &(11, 5), \\
44 &10& 220 &(1, 1), \\
44 &12& 2 &(22, 6), \\
44 &12& 4 &(11, 6), \\
44 &12& 264 &(1, 1), (0, 6), \\
44 &14& 2 &(22, 7), \\
44 &14& 4 &(11, 7), \\
44 &14& 308 &(1, 1), \\
44 &16& 2 &(22, 8), \\
44 &16& 4 &(11, 8), \\
44 &16& 352 &(1, 1), (0, 8), (22, 0), \\
44 &18& 2 &(22, 9), \\
44 &18& 4 &(11, 9), \\
44 &18& 396 &(1, 1), \\
44 &20& 2 &(22, 10), \\
44 &20& 4 &(11, 10), \\
44 &20& 440 &(1, 1), (0, 10), \\
44 &22& 2 &(22, 11), \\
44 &22& 4 &(11, 11), \\
44 &22& 484 &(1, 1), (0, 2), \\
44 &24& 2 &(22, 12), \\
44 &24& 4 &(11, 12), \\
44 &24& 528 &(1, 1), (0, 12), (22, 0), \\
44 &26& 2 &(22, 13), \\
44 &26& 572 &(1, 1), \\
44 &28& 2 &(22, 14), \\
44 &28& 616 &(1, 1), (0, 14), \\
44 &30& 2 &(22, 15), \\
44 &30& 660 &(1, 1), \\
44 &32& 2 &(22, 16), \\
44 &32& 704 &(1, 1), (0, 16), (22, 0), \\
44 &34& 2 &(22, 17), \\
44 &34& 748 &(1, 1), \\
44 &36& 2 &(22, 18), \\
44 &36& 792 &(1, 1), (0, 18), \\
44 &38& 2 &(22, 19), \\
44 &38& 836 &(1, 1), \\
44 &40& 2 &(22, 20), \\
44 &40& 880 &(1, 1), (0, 20), (22, 0), \\
44 &42& 2 &(22, 21), \\
44 &42& 924 &(1, 1), \\
44 &44& 2 &(22, 22), \\
44 &44& 968 &(1, 1), (0, 2), \\
45 &1& 9 &(5, 0), \\
45 &1& 15 &(3, 0), \\
45 &3& 15 &(3, 1), \\
45 &3& 45 &(1, 1), \\
45 &5& 45 &(1, 2), \\
45 &6& 90 &(1, 1), \\
45 &9& 135 &(1, 1), (0, 3), \\
45 &10& 90 &(1, 3), \\
45 &12& 180 &(1, 1), \\
45 &15& 135 &(1, 2), (0, 5), \\
45 &15& 225 &(1, 1), (0, 3), \\
45 &18& 270 &(1, 1), (0, 6), \\
45 &20& 180 &(1, 3), \\
45 &21& 315 &(1, 1), \\
45 &24& 360 &(1, 1), \\
45 &25& 225 &(1, 2), \\
45 &27& 405 &(1, 1), (0, 9), (15, 0), \\
45 &30& 270 &(1, 3), (0, 10), \\
45 &30& 450 &(1, 1), (0, 6), \\
45 &33& 495 &(1, 1), \\
45 &35& 315 &(1, 2), \\
45 &36& 540 &(1, 1), (0, 12), \\
45 &39& 585 &(1, 1), \\
45 &40& 360 &(1, 3), \\
45 &42& 630 &(1, 1), \\
45 &45& 405 &(1, 2), (0, 5), \\
45 &45& 675 &(1, 1), (0, 3), \\
46 &1& 2 &(23, 0), \\
46 &1& 23 &(2, 0), \\
46 &2& 2 &(23, 1), \\
46 &2& 46 &(1, 1), \\
46 &4& 2 &(23, 2), \\
46 &4& 92 &(1, 1), \\
46 &6& 2 &(23, 3), \\
46 &6& 138 &(1, 1), \\
46 &8& 2 &(23, 4), \\
46 &8& 184 &(1, 1), \\
46 &10& 2 &(23, 5), \\
46 &10& 230 &(1, 1), \\
46 &12& 2 &(23, 6), \\
46 &12& 276 &(1, 1), \\
46 &14& 2 &(23, 7), \\
46 &14& 322 &(1, 1), \\
46 &16& 2 &(23, 8), \\
46 &16& 368 &(1, 1), \\
46 &18& 2 &(23, 9), \\
46 &18& 414 &(1, 1), \\
46 &20& 2 &(23, 10), \\
46 &20& 460 &(1, 1), \\
46 &22& 2 &(23, 11), \\
46 &22& 506 &(1, 1), \\
46 &24& 2 &(23, 12), \\
46 &24& 552 &(1, 1), \\
46 &26& 2 &(23, 13), \\
46 &26& 598 &(1, 1), \\
46 &28& 2 &(23, 14), \\
46 &28& 644 &(1, 1), \\
46 &30& 2 &(23, 15), \\
46 &30& 690 &(1, 1), \\
46 &32& 2 &(23, 16), \\
46 &32& 736 &(1, 1), \\
46 &34& 2 &(23, 17), \\
46 &34& 782 &(1, 1), \\
46 &36& 2 &(23, 18), \\
46 &36& 828 &(1, 1), \\
46 &38& 2 &(23, 19), \\
46 &38& 874 &(1, 1), \\
46 &40& 2 &(23, 20), \\
46 &40& 920 &(1, 1), \\
46 &42& 2 &(23, 21), \\
46 &42& 966 &(1, 1), \\
46 &44& 2 &(23, 22), \\
46 &44& 1012 &(1, 1), \\
46 &46& 2 &(23, 23), \\
46 &46& 1058 &(1, 1), (0, 2), \\
48 &1& 2 &(24, 0), \\
48 &1& 4 &(12, 0), \\
48 &1& 8 &(6, 0), \\
48 &1& 16 &(3, 0), \\
48 &1& 6 &(8, 0), \\
48 &1& 12 &(4, 0), \\
48 &1& 24 &(2, 0), \\
48 &2& 2 &(24, 1), \\
48 &2& 4 &(12, 1), \\
48 &2& 8 &(6, 1), \\
48 &2& 16 &(3, 1), \\
48 &2& 6 &(8, 1), \\
48 &2& 12 &(4, 1), \\
48 &2& 24 &(2, 1), \\
48 &2& 48 &(1, 1), \\
48 &3& 6 &(8, 1), \\
48 &3& 12 &(4, 1), \\
48 &3& 24 &(2, 1), \\
48 &3& 48 &(1, 1), \\
48 &4& 2 &(24, 2), \\
48 &4& 4 &(12, 2), \\
48 &4& 8 &(6, 2), \\
48 &4& 6 &(8, 2), \\
48 &4& 96 &(1, 1), (0, 2), \\
48 &6& 2 &(24, 3), \\
48 &6& 4 &(12, 3), \\
48 &6& 8 &(6, 3), \\
48 &6& 6 &(8, 3), \\
48 &6& 96 &(1, 1), (0, 3), \\
48 &6& 144 &(1, 1), (0, 2), \\
48 &8& 2 &(24, 4), \\
48 &8& 4 &(12, 4), \\
48 &8& 192 &(1, 1), (0, 2), \\
48 &9& 144 &(1, 1), \\
48 &10& 2 &(24, 5), \\
48 &10& 4 &(12, 5), \\
48 &10& 240 &(1, 1), \\
48 &12& 2 &(24, 6), \\
48 &12& 4 &(12, 6), \\
48 &12& 192 &(1, 1), (0, 3), \\
48 &12& 288 &(1, 1), (0, 2), \\
48 &14& 2 &(24, 7), \\
48 &14& 4 &(12, 7), \\
48 &14& 336 &(1, 1), \\
48 &15& 240 &(1, 1), \\
48 &16& 2 &(24, 8), \\
48 &16& 4 &(12, 8), \\
48 &16& 96 &(1, 3), (0, 8), \\
48 &16& 384 &(1, 1), (0, 2), \\
48 &18& 2 &(24, 9), \\
48 &18& 4 &(12, 9), \\
48 &18& 288 &(1, 1), (0, 9), \\
48 &18& 432 &(1, 1), (0, 6), (16, 0), \\
48 &20& 2 &(24, 10), \\
48 &20& 4 &(12, 10), \\
48 &20& 480 &(1, 1), (0, 10), \\
48 &21& 336 &(1, 1), \\
48 &22& 2 &(24, 11), \\
48 &22& 4 &(12, 11), \\
48 &22& 528 &(1, 1), \\
48 &24& 2 &(24, 12), \\
48 &24& 4 &(12, 12), \\
48 &24& 384 &(1, 1), (0, 3), \\
48 &24& 144 &(1, 3), (0, 8), \\
48 &24& 576 &(1, 1), (0, 2), \\
48 &26& 2 &(24, 13), \\
48 &26& 624 &(1, 1), \\
48 &27& 432 &(1, 1), \\
48 &28& 2 &(24, 14), \\
48 &28& 672 &(1, 1), (0, 14), \\
48 &30& 2 &(24, 15), \\
48 &30& 480 &(1, 1), (0, 15), \\
48 &30& 720 &(1, 1), (0, 10), \\
48 &32& 2 &(24, 16), \\
48 &32& 192 &(1, 3), (0, 16), (24, 0), \\
48 &32& 768 &(1, 1), (0, 4), (0, 12), (6, 0), \\
48 &33& 528 &(1, 1), \\
48 &34& 2 &(24, 17), \\
48 &34& 816 &(1, 1), \\
48 &36& 2 &(24, 18), \\
48 &36& 576 &(1, 1), (0, 9), \\
48 &36& 864 &(1, 1), (0, 6), (8, 0), (8, 6), \\
48 &38& 2 &(24, 19), \\
48 &38& 912 &(1, 1), \\
48 &39& 624 &(1, 1), \\
48 &40& 2 &(24, 20), \\
48 &40& 240 &(1, 3), \\
48 &40& 960 &(1, 1), (0, 10), \\
48 &42& 2 &(24, 21), \\
48 &42& 672 &(1, 1), (0, 21), \\
48 &42& 1008 &(1, 1), (0, 14), \\
48 &44& 2 &(24, 22), \\
48 &44& 1056 &(1, 1), (0, 22), \\
48 &45& 720 &(1, 1), \\
48 &46& 2 &(24, 23), \\
48 &46& 1104 &(1, 1), \\
48 &48& 2 &(24, 24), \\
48 &48& 768 &(1, 1), (0, 3), \\
48 &48& 288 &(1, 3), (0, 8), \\
48 &48& 1152 &(1, 1), (0, 2), \\
50 &1& 2 &(25, 0), \\
50 &1& 10 &(5, 0), \\
50 &1& 25 &(2, 0), \\
50 &2& 2 &(25, 1), \\
50 &2& 10 &(5, 1), \\
50 &2& 50 &(1, 1), \\
50 &4& 2 &(25, 2), \\
50 &4& 10 &(5, 2), \\
50 &4& 100 &(1, 1), \\
50 &5& 10 &(5, 2), \\
50 &5& 50 &(1, 2), \\
50 &6& 2 &(25, 3), \\
50 &6& 150 &(1, 1), \\
50 &8& 2 &(25, 4), \\
50 &8& 200 &(1, 1), \\
50 &10& 2 &(25, 5), \\
50 &10& 100 &(1, 2), (0, 5), \\
50 &10& 250 &(1, 1), (0, 2), \\
50 &12& 2 &(25, 6), \\
50 &12& 300 &(1, 1), \\
50 &14& 2 &(25, 7), \\
50 &14& 350 &(1, 1), \\
50 &15& 150 &(1, 2), \\
50 &16& 2 &(25, 8), \\
50 &16& 400 &(1, 1), \\
50 &18& 2 &(25, 9), \\
50 &18& 450 &(1, 1), \\
50 &20& 2 &(25, 10), \\
50 &20& 200 &(1, 3), (0, 10), (25, 0), \\
50 &20& 500 &(1, 1), (0, 4), \\
50 &22& 2 &(25, 11), \\
50 &22& 550 &(1, 1), \\
50 &24& 2 &(25, 12), \\
50 &24& 600 &(1, 1), \\
50 &25& 250 &(1, 2), (0, 5), \\
50 &26& 2 &(25, 13), \\
50 &26& 650 &(1, 1), \\
50 &28& 2 &(25, 14), \\
50 &28& 700 &(1, 1), \\
50 &30& 2 &(25, 15), \\
50 &30& 300 &(1, 2), (0, 15), \\
50 &30& 750 &(1, 1), (0, 6), \\
50 &32& 2 &(25, 16), \\
50 &32& 800 &(1, 1), \\
50 &34& 2 &(25, 17), \\
50 &34& 850 &(1, 1), \\
50 &35& 350 &(1, 2), \\
50 &36& 2 &(25, 18), \\
50 &36& 900 &(1, 1), \\
50 &38& 2 &(25, 19), \\
50 &38& 950 &(1, 1), \\
50 &40& 2 &(25, 20), \\
50 &40& 400 &(1, 3), (0, 20), (25, 0), \\
50 &40& 1000 &(1, 1), (0, 8), \\
50 &42& 2 &(25, 21), \\
50 &42& 1050 &(1, 1), \\
50 &44& 2 &(25, 22), \\
50 &44& 1100 &(1, 1), \\
50 &45& 450 &(1, 2), \\
50 &46& 2 &(25, 23), \\
50 &46& 1150 &(1, 1), \\
50 &48& 2 &(25, 24), \\
50 &48& 1200 &(1, 1), \\
50 &50& 2 &(25, 25), \\
50 &50& 500 &(1, 2), (0, 5), \\
50 &50& 1250 &(1, 1), (0, 2), \\
51 &51& 867 &(1, 1), (0, 3), \\
52 &52& 2 &(26, 26), \\
52 &52& 52 &(2, 10), (0, 26), \\
52 &52& 208 &(1, 5), (0, 13), \\
52 &52& 1352 &(1, 1), (0, 2), \\
54 &54& 2 &(27, 27), \\
54 &54& 972 &(1, 1), (0, 3), \\
54 &54& 1458 &(1, 1), (0, 2), \\
55 &55& 605 &(1, 2), (0, 5), \\
56 &56& 2 &(28, 28), \\
56 &56& 392 &(1, 3), (0, 8), \\
56 &56& 1568 &(1, 1), (0, 2), \\
57 &57& 1083 &(1, 1), (0, 3), \\
58 &58& 2 &(29, 29), \\
58 &58& 1682 &(1, 1), (0, 2), \\
60 &60& 2 &(30, 30), \\
60 &60& 720 &(1, 2), (0, 5), \\
60 &60& 1200 &(1, 1), (0, 3), \\
60 &60& 1800 &(1, 1), (0, 2), \\
62 &62& 2 &(31, 31), \\
62 &62& 1922 &(1, 1), (0, 2), \\
63 &63& 189 &(1, 8), (0, 21), \\
63 &63& 1323 &(1, 1), (0, 3), \\
64 &64& 2 &(32, 32), \\
64 &64& 512 &(1, 3), (0, 8), \\
64 &64& 2048 &(1, 1), (0, 2), \\
65 &65& 325 &(1, 5), (0, 13), \\
65 &65& 845 &(1, 2), (0, 5), \\
66 &66& 2 &(33, 33), \\
66 &66& 1452 &(1, 1), (0, 3), \\
66 &66& 2178 &(1, 1), (0, 2), \\
68 &68& 2 &(34, 34), \\
68 &68& 136 &(1, 13), (0, 34), \\
68 &68& 2312 &(1, 1), (0, 2), \\
69 &69& 1587 &(1, 1), (0, 3), \\
70 &70& 2 &(35, 35), \\
70 &70& 980 &(1, 2), (0, 5), \\
70 &70& 2450 &(1, 1), (0, 2), \\
72 &72& 2 &(36, 36), \\
72 &72& 1728 &(1, 1), (0, 3), \\
72 &72& 648 &(1, 3), (0, 8), \\
72 &72& 2592 &(1, 1), (0, 2), \\
74 &74& 2 &(37, 37), \\
74 &74& 2738 &(1, 1), (0, 2), \\
75 &75& 1125 &(1, 2), (0, 5), \\
75 &75& 1875 &(1, 1), (0, 3), \\
76 &76& 2 &(38, 38), \\
76 &76& 2888 &(1, 1), (0, 2), \\
78 &78& 2 &(39, 39), \\
78 &78& 468 &(1, 5), (0, 13), \\
78 &78& 2028 &(1, 1), (0, 3), \\
78 &78& 3042 &(1, 1), (0, 2), \\
80 &80& 2 &(40, 40), \\
80 &80& 1280 &(1, 2), (0, 5), \\
80 &80& 800 &(1, 3), (0, 8), \\
80 &80& 3200 &(1, 1), (0, 2), \\
81 &81& 2187 &(1, 1), (0, 3), \\
82 &82& 2 &(41, 41), \\
82 &82& 3362 &(1, 1), (0, 2), \\
84 &84& 2 &(42, 42), \\
84 &84& 84 &(2, 16), (0, 42), \\
84 &84& 336 &(1, 8), (0, 21), \\
84 &84& 2352 &(1, 1), (0, 3), \\
84 &84& 3528 &(1, 1), (0, 2), \\
85 &85& 1445 &(1, 2), (0, 5), \\
86 &86& 2 &(43, 43), \\
86 &86& 3698 &(1, 1), (0, 2), \\
87 &87& 2523 &(1, 1), (0, 3), \\
88 &88& 2 &(44, 44), \\
88 &88& 968 &(1, 3), (0, 8), \\
88 &88& 3872 &(1, 1), (0, 2), \\
90 &90& 2 &(45, 45), \\
90 &90& 1620 &(1, 2), (0, 5), \\
90 &90& 2700 &(1, 1), (0, 3), \\
90 &90& 4050 &(1, 1), (0, 2), \\
91 &91& 637 &(1, 5), (0, 13), \\
92 &92& 2 &(46, 46), \\
93 &93& 2883 &(1, 1), (0, 3), \\
94 &94& 2 &(47, 47), \\
95 &95& 1805 &(1, 2), (0, 5), \\
96 &96& 2 &(48, 48), \\
96 &96& 3072 &(1, 1), (0, 3), \\
96 &96& 1152 &(1, 3), (0, 8), \\
98 &98& 2 &(49, 49), \\
99 &99& 3267 &(1, 1), (0, 3), \\
100 &100& 2 &(50, 50), \\
100 &100& 2000 &(1, 2), (0, 5), \\
\end{xtabular}

}

\onecolumn

\end{document}